\documentclass[12pt]{amsart}
\usepackage[utf8]{inputenc}
\usepackage{geometry}
\usepackage{verbatim}
\usepackage{amstext}
\usepackage{amsthm}
\usepackage{amssymb}
\usepackage[unicode=true,
 bookmarks=false,
 breaklinks=false,pdfborder={0 0 1},backref=section,colorlinks=false]
 {hyperref}

\makeatletter
\numberwithin{equation}{section}
\numberwithin{figure}{section}
\theoremstyle{plain}
\newtheorem{thm}{\protect\theoremname}[section]
\theoremstyle{definition}
\newtheorem{defn}[thm]{\protect\definitionname}
\theoremstyle{definition}
\newtheorem{problem}[thm]{\protect\problemname}
\theoremstyle{remark}
\newtheorem{rem}[thm]{\protect\remarkname}
\theoremstyle{plain}
\newtheorem{cor}[thm]{\protect\corollaryname}
\theoremstyle{plain}
\newtheorem{lem}[thm]{\protect\lemmaname}
\theoremstyle{definition}
\newtheorem{example}[thm]{\protect\examplename}
\theoremstyle{plain}
\newtheorem{prop}[thm]{\protect\propositionname}

\usepackage{cite}

\include{mypackages}
\include{mymacros}
\usepackage{slashed}
\usepackage[numbers]{natbib}
\usepackage[all]{xy}

\numberwithin{equation}{section}

\newcommand{\Ric}{\operatorname{Ric}}

\newcommand{\C}{\mathbb{C}}
\newcommand{\R}{\mathbb{R}}

\newcommand{\bz}{\bar{z}}
\newcommand{\bw}{\bar{w}}
\newcommand{\Z}{\mathbb{Z}}

\newcommand{\del}{\partial}
\newcommand{\delb}{\bar{\partial}}

\newcommand{\ga}{\alpha}
\newcommand{\gb}{\beta}

\renewcommand{\bar}[1]{\overline{#1}}

\renewcommand{\i}{\sqrt{-1}}

\DeclareMathOperator{\Real}{Re}

\makeatother

\providecommand{\corollaryname}{Corollary}
\providecommand{\definitionname}{Definition}
\providecommand{\examplename}{Example}
\providecommand{\lemmaname}{Lemma}
\providecommand{\problemname}{Problem}
\providecommand{\propositionname}{Proposition}
\providecommand{\remarkname}{Remark}
\providecommand{\theoremname}{Theorem}

\begin{document}
\title[complex surfaces with split tangent]{On canonical metrics of complex surfaces with split tangent and related geometric PDEs }
\author{Hao Fang, Joshua Jordan}
\begin{abstract}
In this paper, we study bi-Hermitian metrics on complex surfaces with split holomorphic tangent bundle and construct 2 types of metric cones. We introduce a new type of fully non-linear geometric PDE on such surfaces and establish smooth solutions. As a geometric application, we solve the prescribed Bismut Ricci problem. In various settings, we obtain canonical metrics on 2 important classes of complex surfaces: primary Hopf surfaces and Inoue surfaces of type $\mathcal{S}_{M}$.
\end{abstract}

\date{\today}

\maketitle
\tableofcontents{}

\section{Introduction}

In this paper, we study Hermitian metrics on complex surfaces with
split tangent bundles. We begin by introducing a structural
condition that  is fundamental to our geometric setup.
\begin{defn}
\label{def:splittangent}A connected complex manifold $(M^{n},I)$ of dimension
$n$ with complex structure $I$ is said to have a \emph{split tangent
bundle, or split tangent, }if 
\[
T^{1,0}M=T^{+}\oplus T^{-},
\]
 where $T^{+}$ and $T^{-}$ are two non-trivial holomorphic sub-bundles. 
\end{defn}

Complex manifolds with split tangent have been extensively studied
in algebraic geometry. In particular, surface examples have
been completely classified by Beauville \cite{Beauville}, where their
universal coverings either have product structures or are Hopf surfaces.
Important examples include Inoue surfaces of type $\mathcal{S}_{M}$
and primary Hopf surfaces, which are two prototypical cases of class
VII manifolds in the Enriques-Kodaira classification theory of complex
surfaces.

In this paper, we explore  metric properties of
surfaces with split tangent. 

\subsection{Notations and setups}

We  begin by introducing some relevant notation and constructions that will be used in this
paper.

First, for a complex manifold $(M,I)$,
with $I$ being a complex structure, an $I$-compatible Hermitian
metric $h$ induces a non-degenerate differential form $\omega=h(I\cdot,\cdot)\in\Lambda^{1,1}(M)$,
which will be called the \emph{Hermitian form of $h$}. In a local
holomorphic coordinate, $\omega= \sqrt{-1}h_{i\bar{j}}dz^{i}\wedge d\bar{z}^{i}$,
where $(h_{i\bar{j}})$ is a positive Hermitian matrix. Through out
this paper, we fix the orientation of $M$ such that an $I$-compatible
Hermitian form $\omega$ is positive. 

For $M$ a complex manifold with split tangent, we define 
\[
\Lambda^{\pm\ p,q}(M)=\Lambda^{p}(T^{\pm,*}(M))\wedge\Lambda^{q}\bar{(T^{\pm,*}(M))}
\]

\begin{defn}
\label{def:metric}For $M$ satisfying Definition \ref{def:splittangent},
we define the space of all $(1,1)$ forms \emph{of split
type} to be
\[
\Lambda^{s}(M):=\{\eta=\eta^{+}+\eta^{-}|\ \eta^{+}\in\Lambda^{+\ 1,1}(M),\ \eta^{-}\in\Lambda^{-\ 1,1}(M)\}
\]
Furthermore, we refer to a Hermitian metric $h$ as \emph{of split
type} if its associated Hermitian form, $\omega$, is of split type.
Locally, we may write the components of a split-type fundamental form
as $\omega^{+}=\i h_{i\bar{j}}^{+}v^{i}\wedge\bar{v}^{j}$ and $\omega^{-}=\i h_{\alpha\bar{\beta}}^{-}w^{i}\wedge\bar{w}^{\beta}$,
where $\{v^{i}\}$ and $\{w^{\alpha}\}$ are bases for $T^{+,*}$
and $T^{-,*}$, respectively. For future use, we define an involution
$\iota:\Lambda^{s}(M)\to\Lambda^{s}(M)$ as 
\begin{equation}
\iota(\eta^{+}+\eta^{-})=\eta^{+}-\eta^{-}.\label{eq:involution-i}
\end{equation}
\end{defn}

Second, we impose  integrability conditions on the Hermitian form.
A Hermitian form $\omega$ is K\"ahler if $d\omega=0$. The K\"ahler
condition has significant topological and geometrical restrictions.
However, important examples of complex surfaces are known to be non-Kähler. In this paper, we first consider the larger class
of\emph{ pluriclosed} metrics instead. A differential form $\omega$ is called\emph{
pluriclosed} if 
\[
\sqrt{-1}\partial\bar{\partial}\omega=0.
\]
Note that we have $\sqrt{-1}\partial\bar{\partial}=\frac{-1}{2}dId$
as a real operator. A well-known result of Gauduchon \cite{gauduchon:nulle},
stated also as Theorem \ref{thm:Gauduchonmetric}, claims that any
Hermitian metric on a compact, complex surface is conformal to a pluriclosed
metric. Therefore, pluri-closedness is a mild constraint in the surface
case.

Third, we introduce a useful operator from bi-complex and generalized
K\"ahler geometry. Since $\partial$: $C^{\infty}(M)\to\Lambda^{1,0}(M)=\Lambda^{+\ 1,0}\oplus\Lambda^{-\ 1,0}$,
we apply proper projections to define $\partial_{\pm}:C^{\infty}(M)\to\Lambda^{\pm\ 1,0}$
such that $\partial=\partial_{+}+\partial_{-}$ \cite{ag:gkwithsplit}
. We have a similar decomposition for $\delb$.
\begin{defn}
\cite[Lemma 7.77]{gfs} For $u\in C^{\infty}(M)$, the \emph{box operator}
$\Box:C^{\infty}(M)\to\Lambda^{s}(M)$ is defined by
\[
\Box:=\i(\del_{+}\delb_{+}-\del_{-}\delb_{-}).
\]
\end{defn}

Fourth, following Streets \cite[c.f. Definition 7.3, 7.4]{s:borninfeld},
we introduce a cohomology group in which important geometric quantities
live. A direct computation shows that $\sqrt{-1}\partial\bar{\partial}\circ\Box u=0$,
which leads to the following chain map:

\begin{equation}
0\to C^{\infty}(M)\overset{\Box}{\longrightarrow}\Lambda^{s}(M)\overset{\i\del\delb}{\longrightarrow}\Lambda^{2,2}(M).\label{eq:add010}
\end{equation}

\begin{defn}
\label{def:splittypeclass}Notation as above. We define the \emph{split
type cohomology group} 
\end{defn}

\[
\mathcal{H}(M)=\{\phi\in\Lambda^{s}(M)|\ \sqrt{-1}\del\delb\phi=0\}/\{\Box u|u\in C^{\infty}(M)\},
\]

and its associated \emph{positive cone} $\mathcal{P\subset\mathcal{H}}$
as

\[
\mathcal{P}(M)=\{[\phi]\in\mathcal{H\ }|\ \exists\ \omega\in[\phi]\text{ s.t. }\omega>0\}.
\]

Fifth, we introduce a space of metric classes following Streets \cite{s:borninfeld}. Whereas, in the K\"ahler
manifold case, any two K\"ahler metrics within the same K\"ahler
class differ by $\sqrt{-1}\del\delb u$, for manifolds with split
tangent bundle, it is natural to study metrics of split type differing
by $\Box u$. More concretely, for a fixed Hermitian metric $\omega_{0}$,
we are interested in metrics of the following form
\[
\omega=\omega_{0}+\Box u>0.
\]
Therefore, the positive cone $\mathcal{P}$ in Definition \ref{def:splittypeclass}
replaces K\"ahler cone in K\"ahler geometry under our setup.

Last, we point out that there exists a dual geometric construction.
For a complex surface $(M,I)$ with split tangent, note that $\Box$
may also be defined for general differential forms instead of smooth
functions. We may then consider \emph{$\Box$-closed Hermitian metrics}.
Notice that similar to (\ref{eq:add010}), for any $u\in C^{\infty}(M)$
and letting $\pi:\Lambda^{1,1}(M)\to\Lambda^{s}(M)$ be the natural
projection,
\[
\Box(\pi(\sqrt{-1}\partial\bar{\partial}u))=0.
\]
We have the following:
\begin{defn}
Notation as above. We define \emph{the second split type cohomology
group} 
\end{defn}

\[
\mathcal{H}'(M)=\{\phi\in\Lambda^{s}(M)|\ \Box\phi=0\}/\{\pi(\sqrt{-1}\partial\bar{\partial}u)|u\in C^{\infty}(M)\},
\]

and the associated \emph{positive cone} $\mathcal{P}'\subset\mathcal{H}'$ 

\[
\mathcal{P}'(M)=\{[\phi]\in\mathcal{H}'\ |\ \exists\ \omega\in[\phi]\text{ s.t. }\omega>0\}.
\]

\subsection{Structure of cohomology and a uniformization theorem}

Now we state the first result of our paper as follows, generalizing a result of Streets and Ustinovskiy \cite[Lemma 5.3]{su:gkrsolitons}:
\begin{thm}
\label{thm: dimension}For a complex surface with split tangent, and
notation set as above, we have $\dim\mathcal{H}=2.$ Also, we have
$\dim\mathcal{H}'=2.$ 
\end{thm}

The second part of Theorem \ref{thm: dimension} is a direct
consequence of its first part, due to the involution constructed in
Definition \ref{def:metric} (see Lemma \ref{lem:involution} for
details). The first part of Theorem \ref{thm: dimension} indicates
that $\mathcal{H},$ which is the tangent space of $\mathcal{P},$
is finite dimensional. Similar results are not known in general non-K\"ahler
Hermitian geometry. Theorem \ref{thm: dimension} indicates our set
up may be viewed as a proper analogue of K\"ahler geometry. For the
proof of Theorem \ref{thm: dimension}, we introduce a simple algebraic
criteria for linear dependence within $\mathcal{H}$. We explore a
linear elliptic PDE of split type to prove Theorem \ref{thm: dimension}
where the dimension being 2, as well as the pluriclosed condition,
plays an important role. 

We use Theorem \ref{thm: dimension} to show a surprising uniformization theorem for pluriclosed bi-Hermitian metrics on Hopf surfaces. 

\begin{defn}[\cite{Kodaira+1975+1471+1510}]
\label{def:Hopf}A primary diagonal Hopf surface $H_{a,b}$, where $\alpha=\Real a>0$, $\beta={\Real b>0}$,
is defined as
\[
H=H_{a,b}=\{(z,w)\in\mathcal{\mathbb{C}}^{2}\backslash(0,0)\}/\sim,
\]
where $(z,w)\sim(e^a z,e^{b}z)$ is an equivalence relation.
\end{defn}
 In \cite{su:gkrsolitons}, Streets-Ustinovskiy constructed a family of Bismut-Ricci soliton metrics, denoted as $\omega^{SU}_t=\omega_t\in \Lambda^s(H_{a,b})$ with $t\in\R$, which are holomophically and metrically equivalent to each other.   For future use, we define \begin{equation}\label{omegaprime2}
     \omega'_t=\frac{d}{dt} \omega^{SU}_t.
 \end{equation} See Section 3 for more details.

 \begin{thm}[Uniformization for Hopf surfaces]Notations as above, we have
    
\label{uniformization}  
\begin{equation}
\mathcal{P}(H_{a,b})=\{[s\omega_{t}],\ s>0,\ t\in\mathbb{R}\}.\label{eq:add123}
\end{equation}
\begin{equation}
    \mathcal{P}'(H_{a,b})=\emptyset.\label{eq:add1234}
\end{equation}
Equivalently, for any bi-Hermitian pluriclosed metric $\omega$ on $H_{a.b}$, there exists a smooth function $u\in C^\infty(H_{a,b})$ such that $\omega_u=\omega+\Box u=s \omega^{SU}_{t}$, where $s>0$. 
\end{thm}

Theorem \ref{uniformization} can be compared to the standard uniformization theorem for $\mathbb S^2$, where any complex metric can be deformed to the standard round metric up to a conformal map. By Theorem \ref{uniformization},  one may use a conformal factor to move  any bi-Hermitian metric on $H_{a,b}$ to a pluriclosed metric, which can then moved by an additional $\Box u$  term to the standard Streets-Ustinovskiy metric. Note that, even for the regular Hopf surface case, where $\alpha=\beta$, Theorem \ref{uniformization} is new. 

In addition, we also have a complete characterization of the metric cone.
\begin{thm} \label{Pstructure}
For $H_{a,b}$,
    $[\Omega]\in \mathcal{P}$ if and only if for any pluriclosed form $\Omega\in[\Omega]$
    $$ \int_H \iota (\omega'_0)\wedge \Omega>0, $$where $\iota $ is the involution given in (\ref{eq:involution-i}) and $\omega'_0$ is defined in (\ref{omegaprime2}). See also  (\ref{omegaprime}) for its explicit form. 
\end{thm}

Theorem \ref{Pstructure} gives a satisfactory intersection-theoretic criterion to determine the bi-Hermitian metric cone, as $\omega'_0$ is an explicit differential form.

\subsection{Prescribing Bismut Ricci}
Next, we discuss the related curvature problem. Note that in the K\"ahler
case, both K\"ahler metrics and their corresponding Kähler Ricci forms live
in $H_{\mathbb{R}}^{1,1}$ of a Kähler manifold. For manifolds with
split tangent, when metrics of split type are considered, both metrics
and their corresponding Bismut Ricci forms live in the split type
cohomology group $\mathcal{H}$ \cite[Lemma 7.77, Proposition 8.20]{gfs}.
More precisely, from any metric $\omega$ of split type, the $(1,1)$
component of its Bismut Ricci form of first type, denoted as $\Ric_{B}^{1,1}$, defines a class
in $[\Ric_{B}^{1,1}]\in\mathcal{H}$ which is independent of $\omega,$
and may be seen as the projection of $c_{1}(TM)\in H^{2}(M)$ into
$\mathcal{H}.$ In particular, following a general result of Streets
\cite[Proposition 8.20]{gfs}, we have 
\begin{equation}
\Ric_{B}^{1,1}(\omega)=-\Box(\log\det\omega^{+}-\log\det\omega^{-}),\label{eq:add012}
\end{equation}
where $\det\omega^{\pm}$ is computed with a local holomorphic coordinate.
See Section 2 for more details. 

Noting the strong similarity between (\ref{eq:add012}) and the well
known Käher-Ricci formula, Streets \cite{s:pcfongkwithsplit} poses the following Calabi problem as
an analogue of the Calabi problem in the Kähler geometry, which was
settled by Yau in his celebrated work \cite{yau:cyproof}. In dimension 2, the problem can be stated as
\begin{problem}
\label{prob:CalabiProblem} Notations as above. For a fixed $[\omega_{0}]\in\mathcal{P},$
and any $\rho\in[\Ric_{B}^{1,1}(\omega_{0})]\in\mathcal{H}$, find a smooth
$\omega_{u}=\omega_{0}+\Box u>0,$ such that $\Ric_{B}^{1,1}(\omega_{u})=\rho$. 
\end{problem}
In this paper, we are able to answer this question in the dimension 2 case.

\begin{thm}
\label{thm:1}If $M$ is a compact, complex surface with split tangent,
and notations are as above. For any fixed $[\omega_{0}]\in\mathcal{P},$
and any $\rho\in[\Ric_{B}^{1,1}]$, there exists a smooth $\omega_{u}=\omega_{0}+\Box u>0,$
such that $\Ric_{B}^{1,1}(\omega_{u})=\rho$. 
\end{thm}

We will comment on the PDE approach to prove~\ref{thm:1} in a more general setting later in the paper.

\subsection{Canonical metrics on Inoue Surfaces of type $\mathcal{S}_{M}$}
\begin{defn}\label{inoue}
For a unimodular matrix $M\in SL(3,\mathbb{Z})$ with eigenvalues
$\alpha\in\mathbb{R}$, $\alpha>1$ and $\beta,\bar{\beta}\in\C.$
Let $a=(a_{1},a_{2},a_{3})$, $b=(b_{1},b_{2},b_{3})$ be the real
and complex eigen-vectors of $M$ with respect to $\alpha$ and $\beta$,
respectively. We have $\alpha|\beta|^{2}=1$. Define a holomorphic action
group $G_{M}$ on $\mathbb{H\times\mathbb{C}}$ generated by 
\begin{align*}
g_{0}(z,w)= & (\alpha z,\beta w),\\
g_{i}(z,w)= & (z+a_{i},w+b_{i}),\quad i=1,2,3.
\end{align*}
An Inoue surface of type $\mathcal{S}_{M}$, denoted as $S$ is defined
as $\mathbb{H\times\mathbb{C}}/G_{M}$. 
\end{defn}
Let $T^+$ be the holomorphic line bundle locally generated by $\frac{\del}{\del z}$ and $T^-$ be the holomorphic line bundle locally generated by $\frac{\del}{\del w}$. $T^\pm$ are both flat  as the corresponding transition functions may be chosen as constants. It is clear that $S$ satisfies Definition \ref{def:splittangent}.  We give the following
\begin{defn}
    The Tricerri metric on the Inoue surface $S$ as above is defined as
    \begin{equation}
        \omega_{a,b}:=a\i\frac{dz\wedge d\bz}{\Im(z)^{2}}+b\i\Im(z)\,dw\wedge d\bw,
    \end{equation}while $a>0,b>0$ and $\Im(z)$
 is the imaginary part of $z.$
\end{defn}
When $a=b=1$, $\omega_{a,b}$ was first constructed by Tricerri \cite{MR0706055}. It is direct to check that these metrics are both $\Box$-closed and
pluriclosed. As a consequence of Theorem \ref{thm: dimension}, these Tricerri metrics can be viewed as representatives of $\mathcal{P}(S)$ and $\mathcal{P}'(S)$ in their respective sense. However, these metrics do not seem to carry curvature properties. 

We explore the special geometric properties of $S$ by the following construction
\begin{defn}
    For $S$ satisfying Definition \ref{inoue}, define a flat holomorphic line bundle 
    \begin{equation}\label{eqn:flatbundleinoue}
        l=(T^+)\otimes(T^-)\otimes (T^-).
    \end{equation}
\end{defn}

Note that any split type Hermitian metric $\omega=\sqrt{-1}(g^{+}dz\wedge d\bar{z}+g^{-}dw\wedge d\bar{w})$ induces a metric $g^l$ on $l$ naturally. In particular, for the local section $\sigma:={( \frac{\del}{\del z})}\otimes {( \frac{\del}{\del w})}\otimes {( \frac{\del}{\del w})} $, its $g^l$ norm  can be computed as
\begin{equation}\label{inducedmetric}
   |\sigma|^2_{g^l}={(g^-)^2}{g^+}. 
\end{equation}
It is now crucial to observe that $|\sigma|^2_{g^l}$ is {\em globally well-defined} on $S$ since $\ga|\gb|^2=1$.

We raise the following question
\begin{problem}{\label{flatmetric}}
    For the flat line bundle $l$ given in \ref{eqn:flatbundleinoue}, find a split type metric $\omega\in [\omega]\in \mathcal{P}'(S)$ such that the Chern connection of its induced metric $g^l$ is flat.
\end{problem} 

Problem \ref{flatmetric} may be seen as a Calabi-Yau type problem.  While a Calabi-Yau manifold $M$ has
a trivial canonical line bundle $K$, the Calabi-Yau metric in any
Kähler class induces a flat metric on $K$. Problem \ref{flatmetric} asks for a canonical flat metric on a flat line bundle.

Since the Chern curvature is computed by $-\i\del\delb \log |\sigma|^2$, by (\ref{inducedmetric}), Problem \ref{flatmetric} can be restated as: Given a $\Box$-closed metric $\omega$, find $\omega_u=\omega+\pi(\i\del\delb u)>0$ and $\xi \in\R$ such that
\begin{equation}
\left(\frac{\omega_{u}^{+}}{\omega_{0}^{+}}\right) \left(\frac{\omega_{u}^{-}}{\omega_{0}^{-}}\right)^2=e^\xi ,\label{eq:addinoue}
\end{equation}

Note that (\ref{eq:addinoue}) is an elliptic Monge-Ampère equation; while the corresponding problem for $\mathcal{P}$ will be hyperbolic. See also   (\ref{eq:add123}).

\begin{thm}
\label{thm:main3}For an Inoue surface $S$ of type $\mathcal{S}_{M}$, given as above, 
\[
\mathcal{P}(S)=\{[\omega_{a,b}]\in\mathcal{H}(S)\ |\ \ a>0,\ b>0.\}
\]
\[
\mathcal{P}'(S)=\{[\omega_{a,b}]\in\mathcal{H}'(S)\ |\ \ a>0,\ b>0.\}
\]
\begin{enumerate}
    
    \item For any pluriclosed split Hermitian metric $\omega$
on $S$, there exists a smooth 
$\omega_{u}=\omega+\Box u=\omega_{a,b}\in{\mathcal{P}}$;  
    
    \item For any $\Box$-closed split Hermitian metric $\omega$
on $S$, there exists a smooth $\omega_{u}=\omega+\pi(\sqrt{-1}\partial\bar{\partial}u)=\omega_{a,b}\in\text{\ensuremath{\mathcal{P}'}(S) ; }$

    \item $\omega_{a,b}$ is the unique metric in its ${\mathcal{P}'}$ class
whose induced metric on $l$ is flat.
\end{enumerate}

\end{thm}

 Theorem \ref{thm:main3} shows that  Tricerri metrics are indeed 
canonical metrics on Inoue surfaces in their respective $\mathcal{P}'$
class when we consider the induced metric on $l.$ Since we have a
complete classification of $\mathcal{P}'$ for Inoue surfaces, we
are not solving the corresponding non-linear PDE. The corresponding problem in $\mathcal{P}$ is clearly hyperbolic, which requires a completely different set of tools to study.

\subsection{A new fully non-linear PDE}
We now turn to the PDE aspect of our paper. With application to special
examples of Hopf surfaces  in mind, we concentrate
on the surface case. In this paper, on a complex surface $M$ with
split tangent, assume that $\omega_{0}\in[\omega]\in\mathcal{P},$
let $\omega_{u}=\omega_{0}+\Box u>0$ be another Hermitian metric.
Assume that $\alpha>0, \beta>0$ and $F\in C^{\infty}(M)$. We consider the
following partial differential equation:
\begin{equation}
\left(\frac{\omega_{u}^{+}}{\omega_{0}^{+}}\right)^{\beta}=e^{F+ \xi}\left(\frac{\omega_{u}^{-}}{\omega_{0}^{-}}\right)^\alpha,\label{eq:add0}
\end{equation}
where $\xi \in\mathcal{\mathbb{R}}$ and $u\in C^{\infty}(M)$ are unknowns.

Equation (\ref{eq:add0}) can be viewed as a Monge-Amp\`ere type
PDE of split type. When $\alpha=\beta,$ it is actually linear. In fact, Theorem \ref{thm:1} is a direct consequence of this case. 
However,
when $\alpha\neq\beta,$ (\ref{eq:add0}) is fully non-linear. There have
been many recent works in the real and complex cases of Monge-Amp\`ere
PDE of split type. See \cite{s:borninfeld,s:globallygenerated,s:pcfongkwithsplit,Streets2014PluriclosedFO,sw:evanskrylov}. 

It is interesting to point out that while there are PDEs exploring
partial components of the Hessian matrix \cite{sw:evanskrylov,Mooney-Savin},
our equation (\ref{eq:add0}), gives non-homogeneous weights to different
components when $\alpha\neq\beta$. As far as we are aware, this is a
new form of geometric Monge-Amp\`ere type equation.

Our main technical result of the paper is the following:
\begin{thm}
\label{thm:main}Let $\alpha, \beta>0$. Assume that $M$ is a closed, complex
surface with split tangent with notation as above. Given any fixed
$[\omega_{0}]\in\mathcal{P},$ and any smooth function $F\in C^{\infty}(M),$
there exists a unique pair of smooth $\omega_{u}=\omega_{0}+\Box u>0$
and $\xi \in\mathbb{R}$ to solve (\ref{eq:add0}).
\end{thm}

We make some remarks regarding the proof of Theorem \ref{thm:main}.
First, it is important to realize that (\ref{eq:add0}) is an elliptic
PDE when $\alpha>0,\beta>0$. Second, since the $\beta=1$ case is essentially linear, its
proof  is similar in spirit to that of Theorem \ref{thm: dimension}.
The general case is much more subtle. Though the dimension is low,
we have to apply many techniques in fully non-linear PDEs with 
interesting new components. We have used
heavily the fact that $\alpha\neq\beta$. We have to obtain both upper
and positive lower bounds of the diagonal terms of the Hessian. Off-diagonal terms of the Hessian have to be estimated separately. In addition, our PDE is not concave, which complicates {\em a priori} estimates further. See Section 4 for more details.

\subsection{New canonical metrics on Hopf surfaces }

With Theorem \ref{thm:main}, we give a new family of \emph{ Calabi-Yau type canonical} metrics on Hopf surfaces utilising the special geometric structure of split tangent.

It is known that the corresponding Bismut-Einstein problem can only
be solved on K\"ahler Calabi-Yau surfaces and regular Hopf surfaces
$H_{a,b}$ with $\alpha=\beta$ as in Definition \ref{def:Hopf}, on which
all Bismut-Ricci flat metrics are unique up to diffeomorphism as quotient
metrics from the standard product metric on $\mathbb{S}^{3}\times\mathbb{R}$
(e.g. \cite[Theorem 8.26]{gfs} or \cite[Theorem 1.4]{STREETS2012366}).
See Section 3 for further discussion. In general, the Streets-Ustinovskiy metric is given as a Bismut-Ricci soliton, which by Theorem \ref{uniformization} may be viewed as the canonical metric on Hopf surfaces.

Again, let $T^+$ be the holomorphic line bundle locally generated by $\frac{\del}{\del z}$ and $T^-$ be the holomorphic line bundle locally generated by $\frac{\del}{\del w}$. $T^\pm$ are both flat bundles over $H$. $H$ satisfies Definition \ref{def:splittangent}.
 
We consider first the regular Hopf surface where $a=b \in\R$. Due to the equivalence
relation, it is easy to see that  $l:=T^+ \otimes (T^-)^*$ is a trivial line bundle. Furthermore,  any Bismut Ricci-flat split metric $\omega$ induces a flat metric
on $l$. In fact, Theorem \ref{thm:1}. shows that the norm of section ${\frac{\del}{\del z}}\otimes ({\frac{\del}{\del w}})^* $
with respect to  the induced metric  is constant. Therefore, $l$ is metrically trivial, too. 

We turn to the case where $\alpha\neq \beta$. While the above construction
no longer works, we may consider a {\em virtual} flat line bundle 
$$l:= (T^+)^\beta\otimes (T^-)^{*\alpha}.$$    
 We note that when $\beta/ \alpha \in\mathbb Q$, $l^{\frac{m}{ \alpha}}$ is well defined over $H$ for a properly chosen integer $m$. Otherwise, $l$ is ill defined. However, the real line bundle $L:=l\otimes \bar{l}$ is well defined over $H$. In other words, if $\sigma:={( \frac{\del}{\del z})}^\beta \otimes {( \frac{\del}{\del w})}^{*\alpha} $, then $\sigma\otimes\bar{\sigma}$ may be viewed as a global non-vanishing section of $L$.
  For any split type Hermitian metric $\omega=\sqrt{-1}(g^{+}dz\wedge d\bar{z}+g^{-}dw\wedge d\bar{w})$ induces a metric $g^L$ on $L$ naturally. It is direct to see that
  \begin{equation}
   |\sigma\otimes\bar{\sigma}|_{g^L}=\frac{{(g^+)^\beta}}{({g^-})^\alpha}. 
\end{equation}

As in the Inoue surface case, we observe that $|\sigma\otimes\bar{\sigma}|_{g^L}$ is now a globally defined function as it is compatible with the equivalence relation used to define $H_{a,b}$. Therefore, we may pose the question of finding flat metrics on $L$, which means that the induced metric has vanishing Chern curvature.
We establish the following:
\begin{thm}
\label{thm:alpha}Notation as above. For any primary Hopf surface
$H_{a,b}$, where $\alpha= \Re(a)>0$, $\beta = \Re(b)>0$ and any fixed $[\omega_{0}]\in\mathcal{P},$
there exists a unique $\omega_{u}=\omega_{0}+\Box u>0,$ whose induced
metric on the trivial real line bundle $L$ is flat.
\end{thm}

\begin{rem}
 The corresponding PDE problem is exactly (\ref{eq:add0}).  Therefore, Theorem \ref{thm:alpha} is a direct consequence of Theorem~\ref{thm:main}.
\end{rem}
\begin{rem}
    Theorem \ref{thm:alpha} asserts the existence of a second kind of canonical pluriclosed
representative for a fixed $[\omega]\in\mathcal{P}$. There are various constructions of pluriclosed metrics on general primary Hopf surfaces (e.g. \cite{su:gkrsolitons, Gauduchon-Ornea}). Our construction is new. It is interesting to compare two kinds of canonical metrics:  Streets-Ustinovskiy metrics and the ones coming from Theorem~\ref{thm:alpha}.
\end{rem}
\begin{rem}
    Since $\mathcal{P}'(H_{a,b})=\emptyset$, we do not have a dual problem in this case. 
\end{rem}

\subsection{Remarks on related works}

Our work is motivated and influenced by many existing works from various
areas. Here we make  comments on some of those. Our list is by
no means comprehensive. Interested readers are referred to works cited
below and references therein. 

In the physics literature, the generalized K\"ahler condition was
introduced by Gates-Hull-Roček\cite{ghr:twistedmultiplet}. Hull-Lindström-Roček-von
Unge-Zabzine identified a scalar equation \cite{Hull2010GeneralizedCM}
in the commuting case, which is related to pluriclosed flow \cite{s:pcfongkwithsplit}
and the PDE that we study in this paper. 

Hitchin in \cite{hitchin:genCY} and subsequently, Gualtieri in \cite{g:gencompgeo}
introduced the notion of generalized complex geometry and generalized
Kähler geometry. Streets and Tian introduced the pluriclosed flow
as an extension of K\"ahler-Ricci flow in \cite{st:pluriclosed}.
Later Street and Tian studied generalized K\"ahler-Ricci flow on
generalized K\"ahler manifolds \cite{STREETS2012366}. This setting
has proved flexible enough for the formulation of several non-K\"ahler
Calabi-Yau conjectures in different regimes. See, for example, \cite[Conjecture 1.1]{apostolov2022generalized},
\cite[Conjecture 1.2]{Apostolov2017TheNG}, and \cite[Conjecture 3.12]{s:pcfongkwithsplit}.
For further background, readers are encourage to refer to the book
by Garcia-Fernandez and Streets \cite[Chapters 7-9]{gfs}. See also Zheng \cite{Zhengsurvey}.

Bott-Chern and Aeppli Cohomologies and their analytical implications
have been extensively studied in non-Kähler complex geometry. See
the note by Schweitzer \cite{Schweitzer2007AutourDL} and the book
by Bismut \cite{bismut} for reference. See also the review by Angella
\cite{angella2015bottchern}. Our construction of split type cohomologies
is first raised in \cite{su:gkrsolitons}. 

For various examples of pluriclosed metrics, we mention the earlier works of  Gauduchon-Ornea
\cite{Gauduchon-Ornea} for a construction of special Gauduchon metrics
on all primary Hopf surfaces. Additionally, Apostolov and Dloussky have constructed pluriclosed metrics on Hopf surfaces which are compatible with two complex structures in \cite{apostolovanddloussky}.  Tian-Streets \cite[Theorem 1.4]{STREETS2012366}
have classified all Bismut Ricci flat metrics in dimension 2. Streets and Ustinovskiy discovered metrics, which we refer to as the Streets-Ustinovskiy metrics, on Hopf surfaces and showed that these are compact, steady generalized K\"ahler Ricci solitons \cite{su:gkrsolitons}; in the commuting case this is equivalent to being a Bismut Ricci soliton. Ye in \cite{Ye2022BismutEM} has proved that all non-Kähler Bismut Einstein metrics are Bismut Ricci flat.
Yang and Zheng \cite{Wang2016OnBF} have classified all Bismut flat
metrics in dimension 2 and 3. We also mention the construction of Tricerri \cite{MR0706055} on Inoue surfaces, which plays a crucial role in this paper. For other references regarding Bismut-Ricci problem, see\cite{s:pcfongkwithsplit,gjs:nonkahler,js:c3}.
%Ye in \cite{Ye2022BismutEM} has proved that all non-Kähler Bismut
%Einstein metrics are Bismut Ricci flat. 

For commuting-type generalized K\"ahler manifolds, Apostolov and
Gualtieri,\cite[Proposition 5]{ag:gkwithsplit} have shown that local
deformations are given by scalar functions. Streets has shown in specific cases that the corresponding cohomology group is finite dimensional \cite[Lemma 5.3]{su:gkrsolitons}, suggesting  a more general theory. 
%On standard Hopf surfaces, Theorem \ref{thm: dimension}
%and a result by Angella-Dloussky-Tomassini \cite{Angella2014OnBC}
%disprove this conjecture. See Remark \ref{rem:add100} %for details.

Monge-Amp\`ere equations have been studied extensively in complex
geometry, initiated by the celebrated works of Yau \cite{yau:cyproof}
on the Calabi conjecture. Many extensions in Hermitian geometry have
been proved by various authors, see \cite{guan2009complex,TW:MAonHerm,TW:MAforpluri,TW:HermformsMA,CTW}
and reference therein.

Split-type elliptic PDEs, as a natural extension of Hessian equations,
have seen some study in the context of both geometry and PDE. See
the works of Streets-Warren \cite{sw:evanskrylov} and Mooney-Savin\cite{Mooney-Savin}.

Several recent studies address metrics on Inoue surfaces after the
earlier work of Tricerri \cite{MR0706055}. See, for example, Shen-Smith
\cite{Shen2022TheCE}, Angella-Tosatti \cite{Angella2021LeafwiseFF}
and Tosatti-Weinkove \cite{TosattiWeinkoveCRF}. 

\subsection{Future problems}

In the following, we raise some future problems.

First, it is a surprising fact that Theorem \ref{thm:main} has two
distinct proofs for cases $\alpha=\beta$ and $\alpha\neq \beta$. It is geometrically
tempting to conjecture that resulting solutions, when properly normalized,
have some continuity property with respect to the parameter $\beta/\alpha$.
However, we have not been able to establish a uniform proof.

It is also our intention to study bi-Hermitian metrics on other families of complex surfaces with split tangent and their properties.

As indicated earlier, a PDE similar to (\ref{eq:addinoue}) 
may be studied directly, which may be useful to analyze the geometry of specific examples.

It is highly likely that  with our existing {\em{a priori}} estimates, a Bismut-Ricci flow approach should give an alternative proof of Theorem \ref{thm:1}, which should carry more geometric consequence than our continuity method approach.

Our results indicate potential constructions in higher dimensions
from both the PDE point of view and the geometric point of view. See \cite{gfs} for background and some interesting problems.

The rest of the paper is organized as follows. In Section 2, we collect
some background definitions and results regarding Gauduchon metrics,
Chern Laplacians, classifications of surfaces with split tangent and some formulas for Bismut Ricci curvature.
In Section 3, we explore the structure of the cohomology groups $\mathcal{H}$ and $\mathcal{H}'$, establish Theorem \ref{thm: dimension} and study metrics of Inoue surfaces. In Section 4, we study the metric cone of Hopf surfaces in detail and establish a uniformization theorem.
In Section 5, we establish
the linear version of Theorem \ref{thm:main} and give an affirmative
answer to Problem \ref{prob:CalabiProblem}. In Section 5, we establish
necessary {\em a priori} estimates to study the non-linear version of (\ref{eq:add0}).
In Section 6, we prove the full version of Theorem \ref{thm:main}
using the continuity method.

\subsection{Acknowledgments}

Both authors would like to thank Connor Mooney, Jeff Streets, Lihe
Wang and Jinyang Wu for discussions. They appreciate valuable comments of Jeff Streets on an earlier draft of this paper. The first named author thanks Bo Guan, Biao Ma, Wei Wei and Fangyang Zheng for discussions. Both authors   acknowledge partial support from NSF RTG grant DMS-2038103.

\section{Relevant Background}

In this section, we collect some background facts and results that
will be used in this paper.

First, for a complex manifold with split tangent, we introduce a
bi-Hermitian structure.
\begin{defn}
\label{def:J}For a complex manifold $(M,I)$ with split tangent,
as given in Definition \ref{def:splittangent}. We define a second
complex structure $J:TM\to TM$ such that $J|_{T^{\pm}}=\pm I.$ 

It is clear that $J^{2}=-\mathrm{Id}.$ A direct computation shows
that $J$ is also integrable. 
\end{defn}

Second, we introduce the notions of \emph{Gauduchon and pluriclosed
Hermitian metrics}.
\begin{defn}
A Hermitian metric $\omega$ on a complex manifold $(M^{n},I)$ is
called \emph{Gauduchon} if $\i\del\delb\omega^{n-1}=0$. When $n=2$,
the Gauduchon condition is equivalent to \emph{the pluriclosed condition.}
We will use these notions interchangeably.
\end{defn}

One reason for considering such metrics is that every conformal class
on a compact, complex manifold which contains a Hermitian metric admits
a Hermitian metric satisfying the Gauduchon condition \cite{gauduchon:nulle}.
We state this theorem in detail below as we will use it extensively
in this work.
\begin{thm}
\label{thm:Gauduchonmetric}\cite{gauduchon:nulle} Let $M^{n}$ be
a compact, complex manifold. Then, if $\omega$ is an arbitrary Hermitian
metric, there is a unique smooth function $f\in C^{\infty}(M)$, such
that
\begin{align}
\i\del\delb(e^{(n-1)f}\omega^{n-1}) & =0,\label{eq:gauduchoneqn}\\
\int_{M}e^{nf}\omega^{n} & =\int_{M}\omega^{n}.
\end{align}
\end{thm}

We list a couple properties of (\ref{eq:gauduchoneqn}) for future
use.
\begin{lem}
\label{lem:fconst} (c.f.\cite{gauduchon:nulle}) Suppose $\omega>0$
is Gauduchon. Let $f\in C^{\infty}(M)$ satisfy $\i\del\delb(f\omega^{n-1})=0$,
then $f\equiv const$.
\end{lem}

\begin{proof}
Adding a sufficiently large constant $C>0$, we may assume that $f+C>0$.
We have
\[
\i\del\delb((f+C)\omega^{n-1})=0.
\]
Thus by Theorem \ref{thm:Gauduchonmetric}, $f+C$ is constant, which
leads to our conclusion.
\end{proof}
We state a technical result for future use.
\begin{lem}
\label{lem:Gauduchonfactorcont} With notation as in Theorem \ref{thm:Gauduchonmetric},
the solution $f$ of (\ref{eq:gauduchoneqn}) depends continuously
on the choice of metric $\omega$ when $\omega\mapsto f$ is considered
as a map $C^{3}(\Lambda^{1,1}(M))\to C^{3}(M)$.
\end{lem}

\begin{proof}
We denote
\[
C_{0}^{1}(\Lambda^{n,n})=\{\eta\in C^{1}(\Lambda^{n,n})|\int_{M}\eta=0\}.
\]
We consider the following map 
\begin{align*}
G:C^{3}(M)\times C^{3}(\Lambda^{1,1}) & \to C_{0}^{1}(\Lambda^{n,n})\times\mathbb{R}\\
(f,\omega) & \mapsto(\i\del\delb(e^{(n-1)f}\omega^{n-1}),\frac{1}{\int_{M}\omega^{n}}\int_{M}e^{nf}\omega^{n}-1).
\end{align*}
Suppose that $G(f_{0},\omega_{0})=0$ for some $(f_{0},\omega_{0})\in C^{3}(\Lambda^{1,1})\times C^{3}(M)$,
then the linearized map is given as follows:
\begin{align*}
\delta G|_{(f_{0},\omega_{0})}(\delta f,0) & =(\i\del\delb(\delta fe^{(n-1)f_{0}}\omega_{0}^{n-1}),\frac{n}{\int_{M}\omega_{0}^{n}}\int_{M}\delta fe^{nf_{0}}\omega_{0}^{n}).
\end{align*}
We claim that $\delta G|_{(f_{0},\omega_{0})}(*,0)$ is bijective. 

We consider the injectivity first. Note that $\delta G|_{(f_{0},\omega_{0})}(\delta f,0)=0$
is equivalent to the following
\begin{equation}
\begin{cases}
\i\del\delb(\delta fe^{(n-1)f_{0}}\omega_{0}^{n-1})=0,\\
\int_{M}\delta fe^{nf_{0}}\omega_{0}^{n}=0.
\end{cases}\label{eq:injec}
\end{equation}
Since $G(f_{0},\omega_{0})=0$, $e^{f_{0}}\omega_{0}$ is Gauduchon.
Therefore we apply Lemma \ref{lem:fconst} to the first equation of
(\ref{eq:injec}) to see that $\delta f$ is constant. Then, by the
second equation of (\ref{eq:injec}), we have concluded that $\delta f\equiv0.$ 

We then consider the surjectivity. Let $ge^{nf_{0}}\frac{\omega_{0}^{n}}{n!}\in C_{0}^{1}(\Lambda^{n,n})$
and $r\in\mathbb{R}$ and consider the following linear elliptic equation:
\begin{equation}
\begin{cases}
\i\del\delb(\delta fe^{(n-1)f_{0}}\omega_{0}^{n-1})=ge^{nf_{0}}\frac{\omega_{0}^{n}}{n!},\\
\frac{1}{\int_{M}\omega_{0}^{n}}\int_{M}\delta fe^{\frac{n}{n-1}f_{0}}\omega_{0}^{n}=r.
\end{cases}\label{eq:surjec}
\end{equation}
As the right hand side of the first equation of (\ref{eq:surjec})
has $ge^{nf_{0}}\frac{\omega_{0}^{n}}{n!}\in C_{0}^{1}(\Lambda^{n,n})$,
we have $\int_{M}g(e^{f_{0}}\omega_{0})^{n}=0.$ Therefore, the first
part of (\ref{eq:surjec}) can be solved by standard elliptic theory.
Pick any such solution $\hat{\delta f}$, then we may re-normalize
this to 
\[
\delta f=\hat{\delta f}+\left(r-\frac{1}{\int_{M}\omega_{0}^{n}}\int_{M}\hat{\delta f}e^{\frac{n}{n-1}f_{0}}\omega_{0}^{n}\right),
\]
which satisfies the second equation of (\ref{eq:surjec}). We have
thus established the surjectivity. 

Now that we have established the claim that the linearized operator
is bijective, we may apply the Implicit Function Theorem for Banach
Spaces \cite[Theorem 17.6]{gt} to show that the implicit function
$\omega\mapsto f_{\omega},$ defined by $G(f_{\omega},\omega)=0$
is continuous.
\end{proof}
\begin{defn}
Notation as above. We define the \emph{Chern Laplacian} as 
\[
\Delta_{\omega}u:=\frac{\i\del\delb u\wedge\omega^{n-1}}{\omega^{n}}.
\]
\end{defn}

The Poisson equation for the Chern Laplacian has a well-understood
classical solvability theory coming from the Fredholm alternative.
\begin{thm}
\label{thm:19}\cite{gauduchon:nulle,CTW} Let $\omega$ be an arbitrary
Hermitian metric on $M^{n}$ with Gauduchon factor $e^{f}$. Then,
for $v\in C^{\infty}(M)$, the Chern-Poisson equation, 
\[
\Delta_{\omega}u=v
\]
admits a unique smooth solution $u$ if and only if 
\[
\int_{M}ve^{(n-1)f}\omega^{n}=0.
\]
\end{thm}

Furthermore, by standard PDE techniques (see \cite[Appendix A]{Alesker2013309}),
the Chern Laplacian admits a Green's function with a lower bound.
\begin{thm}
\label{thm:Green}\cite[Appendix A]{Alesker2013309} For any second-order,
elliptic operator $A$ on a compact manifold $M$, there exists a
Green's function $G:M\times M\to\mathbb{R}$ satisfying 
\begin{enumerate}
\item $\int_{M}G(x,y)A\phi(y)dV(y)=\phi(x)-\frac{1}{|M|}\int_{M}\phi dV$
for all $\phi\in C^{\infty}(M)$,
\item $G(x,y)$ is smooth outside the diagonal $\Delta\subset M\times M$,
\item $G(x,y)\geq-D_{1}$ a.e. on $M\times M$ for some constant $D_{1}>0$,
and
\item For any fixed $x\in M$, $\|G(x,\cdot)\|_{L^{1}(M)}<D_{2}$ for a
constant $D_{2}>0$.
\end{enumerate}
\end{thm}

Next, we state an algebro-geometric result due to Beauville which,
in our language, gives a classification of the compact, complex surfaces
which admit a non-trivially split tangent bundle. This result indicates many non-trivial examples. In addition, it implies the existence of a global choice of coordinates.
\begin{thm}
\label{thm:21}\cite[Theorem C]{Beauville} Let $M$ be a compact
complex surface. $M$ has a nontrivial split tangent bundle if and
only if one of the following occurs:
\begin{itemize}
\item The universal covering space of $M$ is a product of two simply-connected
Riemann surfaces $S_{1}\times S_{2}$ and the group $\pi_{1}(M)$
acts diagonally on $S_{1}\times S_{2}$; in this case, the splitting
of $T_{\C}M$ lifts to the direct sum decomposition $T_{\C}(U\times V)=T_{\C}U\oplus T_{\C}V$.
\item $M$ is a Hopf surface, with the universal covering space $\C^{2}\setminus \{(0,0)\}$.
It has $\pi_{1}(M)\cong\Z\oplus\Z/m\Z$, for some integer $m\geq1$;
it is generated by diagonal automorphisms $(z,w)\mapsto(e^{a}z,e^{b}w)$
with $a$, $b\in\C$, $\Re (a) >0$, $\Re (b)>0$ and $(z,w)\mapsto(\lambda z,\mu w)$ where $\lambda$
and $\mu$ are primitive $m$-th roots of unity.
\end{itemize}
\end{thm}

A direct consequence is the following:
\begin{lem}
\label{lem:localcoordinate}\cite[Lemma 2.3]{s:borninfeld} If $M$
is a complex surface with split tangent, then locally there exists
holomorphic coordinate $(z,w)\in U\times V\subset\mathcal{\mathbb{C\times\mathbb{C}}}$
such that locally $T^{+}=\mathrm{span\{\frac{\partial}{\partial z}\}}$
and $T^{-}=\mathrm{span}\{\frac{\partial}{\partial w}\}$. 
\end{lem}

Furthermore, we also list a theorem of Apostolov and Gualtieri, which,
in our setting, further classifies surfaces with split tangent up
to bi-holomorphism.
\begin{thm}
\cite[c.f. Theorem 1]{ag:gkwithsplit} Any compact, complex surface
$M$ with split tangent is biholomorphic to one of the following:
\begin{enumerate}
\item[(a)] a geometrically ruled complex surface which is the projectivization
of a projectively flat holomorphic vector bundle over a compact Riemann
surface;
\item[(b)] a bi-elliptic complex surface, i.e. a complex surface finitely covered
by a complex torus;
\item[(c)] a compact complex surface of Kodaira dimension 1 and even first Betti
number, which is an elliptic fiberation over a compact Riemann surface,
whose only singular fibers are multiple smooth elliptic curves;
\item[(d)] a compact complex surface of general type, uniformized by the product
of two hyperbolic planes $\mathbb{H}\times\mathbb{H}$ and with fundamental
group acting diagonally on the factors;
\item[(e)] a quotient of a primary Hopf surface by a finite cyclic group action; or
\item[(f)] an Inoue surface of type $\mathcal{S}_{M}$.
\end{enumerate}
All of these cases admit pluriclosed metrics of split type.
\end{thm}

\begin{rem}
Note that all of these manifolds admit K\"ahler metrics except (e)
and (f).
\end{rem}

In the rest of this section we recall the definition of the Bismut connection
and curvature of a Hermitian metric. Let $(M,I,h)$ be a Hermitian
manifold. Let $\omega$ be the corresponding Hermitian form. Let $\nabla^{LC}$
be the Levi-Civita connection of the Riemannian metric induced by
$h$. Bismut introduced the following connection which is compatible
with $h$ and $I,$ with a totally skew-symmetric torsion form.
\begin{defn}
\label{def:Bismut-data} \cite{bismut:localindexthm} For a pluriclosed
manifold $(M,I,h)$ with notations as above, the \emph{Bismut connection}
is defined as
\[
\nabla^{B}=\nabla^{LC}+\frac{1}{2}h^{-1}d^{c}\omega,
\]
where $d^{c}=\i(\partial-\bar{\partial})$. The curvature
tensor induced by $\nabla^{B}$ and the corresponding Ricci form of the first type are denoted
as $R_{B}\in\Lambda^{2}\otimes\Lambda^{1,1}$ and $\Ric_{B}=\operatorname{tr}_{g}R^{B}\in\Lambda^{2}(T^{*}M)$,
respectively. 
\end{defn}

\begin{rem}
    Note that in general, $R_B$ does not enjoy as many symmetries as in  Riemannian and K\"ahlerian cases. Therefore, there are several definitions of Ricci curvatures. In this paper, we consider only the particular type shown above.
\end{rem}
\begin{lem}
\label{lem:BismutRicci} Assume that $M$ is a complex surface with
split tangent and use notation as above. Then the $(1,1)$ part of
the Bismut Ricci form, $\Ric_{B}^{1,1}$, has the following local
representation
\begin{equation}
\Ric_{B}^{1,1}=-\Box(\log\det h^{+}-\log\det h^{-}),\label{eq:add88}
\end{equation}
where locally, $\omega^{+}=\i h_{i\bar{j}}^{+}v^{i}\wedge\bar{v}^{j}$
and $\omega^{-}=\i h_{\alpha\bar{\beta}}^{-}w^{i}\wedge\bar{w}^{\beta}$,
$\{v^{i}\}$ and $\{w^{\alpha}\}$ are basis for $T^{+}$ and $T^{-}$,
respectively.
\end{lem}

\begin{proof}
Note that by Lemma \ref{lem:localcoordinate}, $M$ is generalized
Kähler as in \cite{gjs:nonkahler}. Then (\ref{eq:add88}) is essentially
Proposition 8.20 of \cite{gjs:nonkahler}, since curvature computation
is purely local. It also follows from the following formula (2.5)
of Ivanov-Popadopolous\cite{IvanovPapa}, 
\[
\Ric_{B}=\Ric_{C}-dd_{\omega}^{*}\omega,
\]
where $\Ric_{C}$ is the Ricci form of the Chern connection of $\omega$,
$d_{\omega}^{*}$ is the dual operator of $d$ with respect to metric
$\omega$. By choosing proper local coordinates and separating components
carefully using the Hodge star operator with respect to $\omega$,
a long computation leads to (\ref{eq:add88}).
\end{proof}
It is now clear from (\ref{eq:add88}) that $[\Ric_{B}^{1,1}]\in\mathcal{H}$
is invariant of the choice of local coordinate or Hermitian metric.
\begin{rem}
It is interesting to compare our cohomology $\mathcal{H}$ in this
specific setting with the usual Bott-Chern cohomology and Aeppli cohomology.
\end{rem}

\section{\label{sec:H}Structures of $\mathcal{H}$ and $\mathcal{H}'$}

In this section, we focus on the case where $M$ is a compact, complex
surface with  split tangent  and we study cohomology
groups $\mathcal{H}$ and $\mathcal{H}'$ as defined in Definition
\ref{def:splittypeclass}. In particular, we show that both of them
are finite dimensional for compact, complex surfaces with split tangent.
We further consider some representative examples in Theorem \ref{thm:21}. 

First, we define the following
\begin{defn}\label{def:nomix}
    For $M$ a complex surface with split tangent, let
    \begin{align}
        \mathcal{H}^+(M):=&\{[\eta]\in\mathcal{H}(M),\ \exists \eta\in [\eta], \mathrm{s.t.}\ \eta^2>0\}\\
        \mathcal{H}^-(M):=&\{[\eta]\in\mathcal{H}(M),\ \exists \eta\in [\eta], \mathrm{s.t.}\ \eta^+\geq 0, \eta^-\leq 0\}\\ \nonumber\cup&\{[\eta]\in\mathcal{H}(M),\ \exists \eta\in [\eta], \mathrm{s.t.}\ \eta^+\leq 0, \eta^-\geq 0\}.
    \end{align}
\end{defn}
We give the following elementary fact that will be repeated used later.

\begin{lem}\label{nomix}
    Notation as above.  Let $-\mathcal{P}(M)=\{ [-\eta]\in \mathcal{P}(M)\}$
    \begin{align}
      \label{q1}  \mathcal{H}^+(M)= \mathcal{P}(M)\cup  - \mathcal{P}(M),\\
       \label{q2} \mathcal{H}^+(M)\cap \mathcal{H}^-(M)=\emptyset\\
      \label{q3} \mathcal{P}(M)\cap -\mathcal{P}(M)=\emptyset.
    \end{align}
\end{lem}

\begin{proof}
    Equation (\ref{q1}) is straightforward by connectedness of $M$. See Definition \ref{def:splittangent}.
    
    Assume that (\ref{q2}) is not true, then there exists $[\eta]\in \mathcal{H}^+(M)\cap \mathcal{H}^-(M)$ which means there exists $\eta\in[\eta]$ such that $\eta=\eta^+ +\eta^-$ where $\eta^+$ and $\eta^-$ have different signs. In addition, there exists a smooth function $u$ such that either $\eta+\Box u > 0$ or $\eta +\Box u<0$. Consider the first case where $\eta^+\geq 0$, $\eta^-\leq 0$, $\eta^+ +\i\del_+\delb_+ u >0$ and $\eta^- +\i\del_-\delb_- u >0$, which indicates $\i\del_-\delb_- u >0$.   When considering the point $p\in M$ when $u$ achieves its maximum,    $\i\del_-\delb_- u\leq 0$. We clearly have a contradiction. The other cases are similar and we omit the detail here. We have proved (\ref{q2}).
    
    Equation (\ref{q3}) is proved similarly.
\end{proof}

Next, recall the involution $\iota$ defined in \ref{eq:involution-i}.
\begin{lem}
\label{lem:involution}Notation as above. We have
$$\iota(\Box u)=\pi(\i\del\delb u).$$
Also, the involution $\iota$
 induces an isomorphism, between
$\mathcal{\mathcal{H}}$ and $\mathcal{H}'$, which will also be denoted
as $\iota$. 
\end{lem}

This is straightforward and we omit the proof here. Additionally,
we will briefly note that a complex surface $M$, viewed as a Hermitian
manifold with a metric compatible to $I$, and the same smooth manifold
$M$, viewed as a Hermitian manifold with a metric compatible to $J$,
have opposite choices of orientation. Since $\mathcal{P}$ and $\mathcal{P}'$
are both defined with respect to $I,$ $\iota$ does not induce an
isomorphism between $\mathcal{P}$ and $\mathcal{P}'$.

We are ready to introduce an algebraic operation that characterizes linear
relations within $\mathcal{H}$.
\begin{defn}
\label{def:bracket}Define the following global  anti-symmetric, bi-linear
bracket on pluriclosed differential forms of split type:
\begin{equation}\label{bracket}  
\{\eta,\gamma\}=\int_M \iota(\eta)\wedge\gamma=\int_{M}\eta_{+}\wedge\gamma_{-}-\eta_{-}\wedge\gamma_{+}.
\end{equation}
\end{defn}
It is direct to check that for any pluriclosed split form $\eta$ and smooth
function $u$,
\[
\{\Box u,\eta\}=\int_M \i\del\delb u\wedge \eta=0
\]
Therefore, this anti-symmetric, bi-linear bracket descends to $\mathcal{H}$-classes, offering the
following criterion to distinguish different classes in $\mathcal{H}.$ 
\begin{thm}
\label{t:perpcriterion}For a compact, complex surface having split
tangent bundle, if $[\omega_{1}]\in\mathcal{H}$ and $[\omega_{2}]\in\mathcal{P}$,
then $[\omega_{1}]=c[\omega_{2}]$ if and only if $\{[\omega_{1}],[\omega_{2}]\}=0$.
\end{thm}

\begin{proof}
We use the notation $\omega_{u}=\omega_{1}+\Box u$. We consider the
equation
\begin{equation}
\omega_{u}^{+}\wedge\omega_{2}^{-}=\omega_{u}^{-}\wedge\omega_{2}^{+}.\label{eq:spanning}
\end{equation}
This can be rearranged as a Chern-Poisson equation.
\begin{equation}
\i\del\delb u\wedge\omega_{2}=\omega_{2}^{+}\wedge\omega_{1}^{-}-\omega_{2}^{-}\wedge\omega_{1}^{+}\label{eq:spanning2}
\end{equation}
Since $\omega_{2}$ is pluriclosed, hence it is Gauduchon. By Theorem
\ref{thm:19}, (\ref{eq:spanning2}) is solvable if and only if its
right-hand side integrates to zero over $M$, which is equivalent
to
\begin{equation}
\{\omega_{1},\omega_{2}\}=0.\label{eq:spanning3}
\end{equation}
Therefore, if (\ref{eq:spanning3}) holds, since $\omega_{2}>0,$
there exist functions $f_{\pm}>0$ such that $\omega_{u}^{\pm}=f_{\pm}\omega_{2}^{\pm}$.
Therefore, (\ref{eq:spanning}) becomes 
\begin{equation}
f_{+}=f_{-}.\label{eq:spanning4}
\end{equation}
Thus, $\omega_{u}=f_{+}\omega_{2}$. However, by Lemma \ref{lem:fconst}
and the fact that $\omega_{2}>0$ is Gauduchon, $f_{+}$ must be constant.
We have proved that $\omega_{1}+\Box u=c\omega_{2}$.

The other direction of the equivalence relation is straightforward
and we omit its proof.
\end{proof}
Notice, that for a given compact, complex surface with split tangent
bundle, Theorem \ref{t:perpcriterion} indicates that, as in the case
of the Calabi-Yau theorem, one can only hope for uniqueness to hold
within a cohomology class. To see this, we construct a one parameter
family of pluriclosed metrics with the same $\mathbb{\mathrm{Ric}}_{B}^{1,1}(M)$.
For a fixed pluriclosed metric $\omega_{0}$, let 
\begin{equation}
\tilde{\omega}_{t}=\exp(t)\omega_{0}^{+}+\exp(-t)\omega_{0}^{-}.\label{eq:add1}
\end{equation}
 By Theorem \ref{thm:Gauduchonmetric}, there exists $f_{t}$ smooth
such that 
\begin{equation}
\omega_{t}=\exp(f_{t})\tilde{\omega_{t}}\label{eq:add2}
\end{equation}
 is pluriclosed and $\int_{M}\omega_{t}^{2}$=1 for all $t$. It is
direct to check that this family of pluriclosed metrics have identical
Bismut Ricci curvature (see Equation \ref{eq:add88}). However, we
have the following observation: if $t>s,$ then by (\ref{eq:add1}),
\begin{equation}
\{\omega_{t},\omega_{s}\}=\{\exp(f_{t})\tilde{\omega_{t}},\exp(f_{s})\tilde{\omega_{s}}\}>0.\label{eq:add3}
\end{equation}
In other words, Theorem \ref{t:perpcriterion} implies $\omega_{t}$
and $\omega_{s}$ belong to different classes in the projectivization
of $\mathcal{H}$. 

Finally, we are ready to prove Theorem \ref{thm: dimension}. Due
to Lemma \ref{lem:involution}, it is sufficient to prove the following:
\begin{thm}
\label{t:dim2}Suppose that $M^{2}$ is a compact, complex surface
with split tangent bundle, and notation as above, then $\dim_{\R}\mathcal{H}=2$.
\end{thm}

\begin{proof}
We start with any split Hermitian metric $\omega_{0}$. As conformal
transformations preserve split forms, we can assume without loss of
generality that $\omega_{0}$ is Gauduchon \cite{gauduchon:nulle}.
We claim that $[\omega_{0}]\in\mathcal{H}$ is non-zero. Otherwise
$\omega_{0}=\Box u$. However, if $p\in M$ were a maximal point of
$u$, it would be the case that $\i\del\delb u(p)\geq0$, i.e. $\omega_{0}$
is not be positive at $p$. This is a contradiction, proving the claim.

Now consider $\omega_{1}$ as given in (\ref{eq:add2}). By (\ref{eq:add3})
and Theorem \ref{t:perpcriterion}, the set $\{[\omega_{0}]_{\mathcal{H}},[\omega_{1}]_{\mathcal{H}}\}$
is linearly independent. Therefore, $\dim_{\R}\mathcal{H}\geq2$.

Let $\omega$ be a pluriclosed, split $(1,1)$-form. We claim that
$[\omega]\in\mathrm{span([\omega_{1}],[\omega_{0}]).}$ If $\{\omega,\omega_{1}\}=0$,
the claim holds due to Theorem \ref{t:perpcriterion}. Otherwise,
by possibly adding a negative sign in front of $\omega,$ we may assume
$\{\omega,\omega_{1}\}<0$. We consider for some $a,b\in\R$, the
following split-type forms
\begin{equation}
\omega_{L}=a\omega_{1}+\omega,\quad\omega_{R}=b\omega_{1}+\omega_{0}.\label{eq:add001}
\end{equation}
The claim will follow from Theorem \ref{t:perpcriterion} if we are
able to find $a$ and $b$ for which
\begin{equation}
\{\omega_{L},\omega_{R}\}=0.\label{eq:perpcrit}
\end{equation}
Note that (\ref{eq:perpcrit}) is equivalent to 
\[
\{\omega_{L},\omega_{R}\}=a\{\omega_{1},\omega_{0}\}+b\{\omega,\omega_{1}\}+\{\omega,\omega_{0}\}
\]
 As $\{\omega,\omega_{1}\}<0$, we may choose $b>0$ sufficiently
large such that
\[
b\{\omega,\omega_{1}\}+\{\omega,\omega_{0}\}<0.
\]
Then, by (\ref{eq:add3}) $\{\omega_{1},\omega_{0}\}>0$, it is possible
to choose $a>0$ by
\[
a=-\frac{b\{\omega,\omega_{1}\}+\{\omega,\omega_{0}\}}{\{\omega_{1},\omega_{0}\}}>0.
\]
As $b>0$, $\omega_{R}>0$ as well. By Theorem \ref{t:perpcriterion},
there is some $u\in C^{\infty}(M)$ and some $c$ such that
\begin{equation}
\omega_{L}+\Box u=c\omega_{R},\label{eq:ceqn}
\end{equation}
which by (\ref{eq:add001}) is equivalent to
\[
\omega=(cb-a)\omega_{1}+c\omega_{2}+\Box\tilde{u},
\]
proving the claim. 
\end{proof}
\begin{rem}
\label{rem:positive} In the proof given above, by taking $a$ large
enough such that $\omega_{L}>0$, we have both $[\omega_L]\in\mathcal{P}$ and $[\omega_R]\in\mathcal{P}$, which indicates that $c>0$ by Lemma \ref{nomix}.
\end{rem}

Finally we  discuss the consequence of our results in light of the classification
result Theorem \ref{thm:21} of Beauville. In particular, we compute
$\mathcal{H}$ explicitly in the following simple example. 
\begin{example}
Let $X=(\Sigma_{+}\times\Sigma_{-},I)$ be a compact complex surface
where $\Sigma_{\pm}$ are compact, simply-connected Riemann surfaces
with complex structures $I_{\pm}$, a splitting of the tangent bundle
induced by the product structure $T^{1,0}X=T^{1,0}\Sigma_{+}\oplus T^{1,0}\Sigma_{-}$,
and product complex structure $I=I_{+}\oplus I_{-}$. Then $X$
has split tangent bundle and $\mathcal{H}$ is clearly spanned by
the classes of the semi-positive forms $\pi_{+}^{*}\omega_{+}$and
$\pi_{-}^{*}\omega_{-}$ given by the pull-backs through the projection
maps $\pi_{\pm}$ of the K\"ahler metrics $\omega_{\pm}$. 
\end{example}

\begin{example}
Given an Inoue surface of type $\mathcal{S_{M}}$, the classes of
the following semi-positive invariant forms generate $\mathcal{H}$
and $\mathcal{H}'$.

\begin{align*}
\omega_{1}= & \i\frac{dz\wedge d\bz}{\Im(z)^{2}},\\
\omega_{2}= & \i\Im(z)\,dw\wedge d\bw.
\end{align*}
\end{example}

We are now ready to prove Theorem \ref{thm:main3}.
\begin{proof}[Proof of  Theorem \ref{thm:main3}]
    We consider $\mathcal{H}(S)$ first. By Theorem \ref{thm: dimension}, for any  pluriclosed split metric $\omega>0$, there exist $a,b\in\mathbb R$ and $u\in C^\infty (S)$ such that
    $$\omega=a\omega_1 +b\omega_2+\Box u.
    $$
    Notice that by Lemma \ref{nomix}, $[a\omega_1 +b\omega_2]\in\mathcal{P}(S)$ if and only if $a>0$ and $b>0$. We have finished the proof. A similar argument may prove the statement regarding $\mathcal{P}'$. Finally, the ellipticity of the corresponding PDE implies the uniqueness of the solution.
\end{proof}
\section{\label{sec:Examples}Uniformization for Hopf Surfaces}

Now we move to the primary Hopf surface $H_{a,b}$ introduced in Definition
\ref{def:Hopf}. $H_{a,b}$ is called standard
if $\alpha=\beta$.

\begin{defn}  \cite{su:gkrsolitons} Given $H_{a,b}$ as above.    The Streets-Ustinovskiy metric is defined for any $t\in\mathbb R$ as
    \begin{equation}\label{SUmetric}
        \omega_t=\i (k(\mu-\nu+t)\frac{dz\wedge d\bz}{\alpha^2|z|^2}+[1-k(\mu-\nu+t)]\frac{dw\wedge d\bw}{\beta^2|w|^2}),
    \end{equation} 
    where $\mu=\frac{\log |z|^2} \alpha $, $\nu=\frac{\log |w|^2}{\beta}$, and $k(x):\mathbb{R}\to(0,1)$ is the strictly monotone increasing function satisfying $k(0)=\frac{1}{2}$ and $$k'(x)=k(1-k)[(\beta-\alpha)k(x)+\ga].$$
\end{defn}
In \cite{su:gkrsolitons}, Streets-Ustinovskiy show that $\omega_t$ are indeed smooth pluriclosed metrics on $H_{a,b}$. 

From Definition \ref{SUmetric}, we also compute
\begin{equation}\label{omegaprime}
    \omega_t':=\frac{d}{dt}  \omega_t=\i k'(\mu-\nu+t) (\frac{dz\wedge d\bz}{\ga^2|z|^2}-\frac{dw\wedge d\bw}{\gb^2|w|^2}),
\end{equation}
which is clearly pluriclosed. For simplicity, we write $\omega_0$ and $\omega'_0$ as $\omega$ and $\omega'$, respectively. It is also direct to compute by (\ref{eq:add88})
\begin{equation}
   \Ric_{B}^{1,1}(\omega_t) 
 =\i (\beta-\alpha )k'(x+t)(\frac{dz\wedge d\bz}{\alpha^2|z|^2}-\frac{dw\wedge d\bw}{\beta^2|w|^2})=(\beta-\alpha)\omega'_t,
\end{equation}
which indicates that $\omega_t$ is a Bismut-Ricci soliton metric when $\alpha\neq\beta$. This is the motivation of its construction by Streets-Ustinovskiy.

 Note that when  $\alpha=\beta$, a direct computation shows  
\begin{align}
\omega & =\frac{\i}{\alpha^2}\frac{dz\wedge d\bz+dw\wedge d\bw}{|z|^{2}+|w|^{2}},\nonumber \\
\omega' & =\frac{\i}{\alpha}\frac{|w|^{2}dz\wedge d\bz-|z|^{2}dw\wedge d\bw}{(|z|^{2}+|w|^{2})^{2}} ,\label{eq:add89}
\end{align}where $\omega$ is, up to a scaling, the standard product metric on $S^3\times S^1$, which is Bismut Ricci flat.

We now collect the following computational results:
\begin{lem}
     For $H_{a,b}$ and forms defined as above, and a positive constant $c=\frac{8\pi^2}{\ga\gb}$, we have
\begin{align}
       \{\omega_t,\omega\} & = c t. \\
           \{\omega_t,\omega'\}&=c.
\end{align}

    \label{integral}
\end{lem}

\begin{proof}
    First, we consider a dense open set in  $ H_{a,b}$ as   $\{(z,w)\in \C^2 | \, zw\neq 0\}/\sim$, which can be lifted by a bijection to the domain $$D'=\{ (z,w)\in\C^2 | \frac{\log |z|^2} \alpha\in \R,\ \frac{\log |w|^2}{\beta}\in (0,2) \}.$$ 
    Clearly $D'=S^1\times S^1\times D''$, where $D''=\{(\mu,\nu)\in \R^2|\ \mu\in\R,\ \nu\in(0,2)\}$.

    Next, we define the following function
    $f(t):=\{\omega_t,\omega\}$ (recall Definition \ref{bracket}), and compute by (\ref{SUmetric}) and standard polar coordinate change formula:
    \begin{align}
    f(t)=\int_H [k(\mu-\nu+t)-k(\mu-\nu)] \frac{dz\wedge d\bz}{\alpha^2|z|^2}\wedge\frac{dw\wedge d\bw}{\beta^2|w|^2 }  \nonumber  \\
    =(\frac{4\pi^2}{\alpha\beta})\int_{D''} [k(\mu-\nu+t)-k(\mu-\nu)] d\mu d\nu \label{mid1}
    \end{align}
    After changing variables from ($\mu,\nu$) to $(x:=\mu-\nu,\nu)$, (\ref{mid1}) may be further reduced to
    \begin{equation}
        f(t)=c\int_{x\in \R}[k(x+t)-k(x)]dx. \label{mid2}
    \end{equation}
    Now we use (\ref{mid1})  and properties of $k$ listed in Definition \ref{SUmetric} to compute
    \begin{equation}
        f'(t)=c\int_{-\infty}^\infty k'(x+t) dx=c, \label{mid3}
    \end{equation}
    Therefore, considering $f(0)=0$, integrating (\ref{mid3}) completes the proof of the first part. The second part is very similar and we omit it here.
\end{proof}

Now we are ready to give the proof of Theorem \ref{uniformization}.

\begin{proof}[Proof of Theorem \ref{uniformization}] By Lemma \ref{integral} and Theorems \ref{thm: dimension} and {\ref{t:perpcriterion}}, we have $$\mathcal{H}(H_{a,b})=\mathrm{span}\{\omega,\omega'\}.$$ In addition, using basic linear algebra and Lemma \ref{integral}, we may show that
 \begin{equation} \label{beijing5}
         [\omega_t]=[\omega+t\omega'].
     \end{equation}
    It is clear now that 
    \begin{equation}
\tilde{\mathcal{P}}(H_{a,b}):=\{[s\omega_{t}],\ s>0,\ t\in\mathbb{R}\}\subset \mathcal{P}(H_{a,b}), \label{eq:add1234567}
\end{equation} Therefore 
\begin{equation}\label{eq:add12345}
    -\tilde{\mathcal{P}}(H_{a,b})\subset - \mathcal{P}(H_{a,b}).
\end{equation} 
From (\ref{beijing5}), (\ref{eq:add1234}), (\ref{eq:add12345}), Lemma \ref{nomix}, and the fact that  $\omega'\in \mathcal{H}^-$, we conclude that
$$ \tilde{\mathcal{P}}(H_{a,b})\cup  -\tilde{\mathcal{P}}(H_{a,b})\cup \mathcal{H}^-= \mathcal{H}.$$
  By Lemma \ref{nomix}, we   conclude that $\tilde{\mathcal{P}}(H_{a,b})= \mathcal{P}(H_{a,b})$.
\end{proof}

We may also give the following
\begin{proof}[Proof of Theorem \ref{Pstructure}]  By Theorem \ref{uniformization}, any $[\Omega]=p[\omega]+q[\omega']\in \mathcal{P}$ if and only if $p>0$. Now note that $p=\{[\Omega],[\omega']\}/\{[\omega],[\omega']\}$, we have the conclusion by checking the definition of the bracket operation (\ref{def:bracket}) and Lemma \ref{integral}.
\end{proof}

\begin{rem}
From the proof above it is clear that the form $\omega'_0$ in Theorem \ref{Pstructure} may be replaced by any $\omega'_t$, or any $\tilde{\omega}\in[\omega']$.
\end{rem}
Finally, we make a remark regarding Aeppli cohomology. Readers should refer to \cite{Angella2014OnBC} for definitions and background knowledge. In general, our split type cohomology will be a larger vector space than $H^{1,1}_A$ due to our construction. In particular, a result of \cite{Angella2014OnBC}  computes the Aeppli cohomology
of Hopf surfaces and indicates that the dimension of $H_{A}^{1,1}$ of a primary
Hopf surface is 1. The following lemma demonstrates the difference between these notions in the special case of a primary Hopf surface, generalizing a result of Streets and Ustinovskiy on standard Hopf surfaces \cite[Lemma 5.3]{su:gkrsolitons}. Using the Aeppli cohomology definition as in \cite{Angella2014OnBC}, we have:
\begin{lem}
\label{lem:aeppli}Notation as above.  $\omega'$ defined in (\ref{omegaprime})
is cohomologous to $0$ in Aeppli cohomology.
\end{lem}

\begin{proof}
We consider the family of complex diffeomorphisms $\phi_{t}(z,w)=(e^{\frac{\alpha t}{4}}z,e^{-\frac{\beta t}{4}}w)$,
which is generated by a holomorphic vector field $X=\frac{\alpha z}{4}\frac{\del}{\del z}-\frac{\beta w}{4}\frac{\del}{\del w}$.
Then, a direct computation shows that $\phi_t^*(\mu-\nu)=\mu-\nu+t$. We can then compute the Lie derivative using the formula 
\begin{equation}\label{eqn:lieder}
\mathcal{L}_X\omega = \frac{d}{dt}|_{t=0}\phi_t^*\omega. 
\end{equation}
But, notice that $\phi_t^*\omega=\omega_t$, so that (\ref{eqn:lieder}) becomes
$$\mathcal{L}_X\omega=\omega'.$$
Letting $\gamma=i_{X}\omega\in\Lambda^{0,1}$. One has 
\[
\omega'=\del\gamma+\delb\bar{\gamma}.
\]
In other words, $\omega'\in[0]\in H_{A}^{1,1}$. 
\end{proof}
%\begin{rem}
%\label{rem:add100}By Example \ref{cor:spanningset} and %Lemma \ref{lem:aeppli},
%that $[\omega']_{\mathcal{H}}\neq0$ and %$[\omega']_{A}=0.$ This
%example answers negatively a question posed in \cite[Question 7.78]{gfs}.
%\end{rem}

\section{Split-type PDE and Prescribing Bismut-Ricci}

In this section, we prove the linear case of our main PDE result,
a classical solvability theorem for the twisted Monge-Amp\`ere equation
in dimension 2. Following  \cite{TW:MAonHerm,TW:HermformsMA,TW:MAforpluri}
we define \emph{a solution of the twisted Monge-Amp\`ere equation}
to be a pair $(u,\xi)\in C^{\infty}(M)\times\mathbb{R}$ 
 solving the
twisted Monge-Amp\`ere equation, where $\xi$ is a  real
parameter. In particular, our main theorem of the section is the following
special case of Theorem \ref{thm:main}: 
\begin{thm}
\label{t:TMAsolved} Let $[\omega_{0}]\in\mathcal{P}$. For any $F\in C^{\infty}(M)$,
there exists a unique pair $u\in C^{\infty}(M)$ and $\xi \in\mathbb{R}$
solving\\
\begin{equation}\label{eq:TMA1}
\begin{cases}
\frac{\omega_u^+}{\omega_0^+}= e^{F+\xi} \frac{\omega_u^-}{\omega_0^-},\\
\omega_u=\omega_0 + \Box u > 0,\\
\min_M u =0.
\end{cases}
\end{equation}
\end{thm}

\begin{proof}
We rewrite the equation as 
\[
\omega_{0}^{+}\wedge\omega_{0}^{-}+\i\del_{+}\delb_{+}u\wedge\omega_{0}^{-}=e^{F+\xi}\omega_{0}^{+}\wedge\omega_{0}^{-}-e^{F+\xi}\i\del_{-}\delb_{-}u\wedge\omega_{0}^{+},
\]
or
\begin{equation}
\i\del\delb u\wedge(e^{F+\xi}\omega_{0}^{+}+\omega_{0}^{-})=(e^{F+\xi}-1)\omega_{0}^{+}\wedge\omega_{0}^{-}.\label{eq:add4}
\end{equation}
 Let $\tilde{\omega}=e^{F+\xi}\omega_{0}^{+}+\omega_{0}^{-}$, (\ref{eq:add4})
can be viewed as a Chern-Poisson equation, 
\begin{equation}
\Delta_{\tilde{\omega}}u=(1-e^{-F-\xi}).\label{eq:add5}
\end{equation}
By Theorem \ref{thm:Gauduchonmetric}, there is a unique smooth function
$f_{F,\xi}$ depending on $F$ and $\xi$, such that
\[
\int_{M}e^{2f_{F,\xi}}\tilde{\omega}^{2}=\int_{M}\tilde{\omega}^{2},\quad\i\del\delb(e^{f_{F,\xi}}\tilde{\omega})=0.
\]
Therefore, we can rewrite (\ref{eq:add4}) in terms of the pluriclosed
metric $\hat{\omega}=e^{f_{F,\xi}}\tilde{\omega}$. 
\begin{equation}
\Delta_{\hat{\omega}}u=e^{-f_{F,\xi}}(1-e^{-F-\xi})\label{eq:add6}
\end{equation}
By Theorem \ref{thm:19}, (\ref{eq:add6}) is solvable if and only
if
\[
0=\int_{M}e^{-f_{F,\xi}}(1-e^{-F-\xi})\hat{\omega}^{2}=\int_{M}e^{f_{F,\xi}}(1-e^{-F-\xi})\tilde{\omega}^{2}=\int_{M}e^{f_{F,\xi}}(e^{F+\xi}-1)\omega_{0}^{2}.
\]
Note that for any $\xi_{1}>-\inf F$, 
\[
\int_{M}e^{f_{F,\xi_{1}}}(e^{F+\xi_{1}}-1)\omega_{0}^{2}>0.
\]
Similarly, for any $\xi_{2}<-\sup F$, we have 
\[
\int_{M}e^{f_{F,\xi_{2}}}(e^{F+\xi_{2}}-1)\omega_{0}^{2}<0.
\]
We then apply Lemma \ref{lem:Gauduchonfactorcont} and the intermediate
value theorem to determine that there exists a $\xi\in(\xi_{2},\xi_{1})$
such that 
\[
\int_{M}e^{f_{F,\xi}}(e^{F+\xi}-1)\omega_{0}^{2}=0.
\]
Therefore, we have established the existence of a smooth solution
for (\ref{eq:add6}). The positivity of the resulting $\omega_{u}$
follows immediately from the argument in Remark \ref{rem:positive}. 

To see the uniqueness, suppose that $(u,\xi_{1})$ and $(v,\xi_{2})$
are both solutions to (\ref{eq:TMA1}). Then the following holds
\begin{equation}
\frac{\omega_{u}^{+}\wedge\omega_{0}^{-}}{\omega_{0}^{+}\wedge\omega_{u}^{-}}=e^{\xi_{1}-\xi_{2}}\frac{\omega_{v}^{+}\wedge\omega_{0}^{-}}{\omega_{0}^{+}\wedge\omega_{v}^{-}},\label{eq:uniqueness}
\end{equation}
which implies that 
\begin{equation}
\omega_{u}=e^{f}(e^{\xi_{1}}\omega_{v}^{+}+e^{\xi_{2}}\omega_{v}^{-})\label{eq:add411}
\end{equation}
for some smooth function $f$. Since $\omega_{u}=\omega_{v}+\Box(u-v)$,
by Theorem \ref{t:perpcriterion}, we have
\[
0=\{\omega_{v},e^{f}(e^{\xi_{1}}\omega_{v}^{+}+e^{\xi_{2}}\omega_{v}^{-})\}=\int_{M}(e^{\xi_{2}}-e^{\xi_{1}})e^{f}\omega_{v}^{2}.
\]
Therefore, $\xi_{1}=\xi_{2}$. Since $\omega_{u}$ and $\omega_{v}$ are
both pluriclosed, applying Lemma \ref{lem:fconst} to (\ref{eq:uniqueness})
indicates that $f\equiv const$. Therefore, (\ref{eq:uniqueness})
indicates that $\omega_{v}+\Box(u-v)=c\omega_{v}$. The uniqueness
now follows immediately if $c=1$. When $c\neq1$ we find
\[
\omega_{v}=\Box\left(\frac{v-u}{1-c}\right),
\]
which indicates that $[0]_{\mathcal{H}}\in\mathcal{P}$, which is absurd. See also the proof of Lemma \ref{nomix}. We have thus finished the proof.
\end{proof}

\section{A Priori Estimates For Fully Nonlinear PDEs}

In this section, we consider the (\ref{eq:add0}) in its full generality.
We first simplify (\ref{eq:add0}) on a compact complex surface with split tangent to the following:
\begin{equation}
\begin{cases}
(\frac{\omega_{u}^{+}}{\omega_{0}^{+}})^{\beta}=e^{F+\xi}(\frac{\omega_{u}^{-}}{\omega_{0}^{-}}),\\
\omega_{u}=\omega_{0}+\Box u>0.
\end{cases}.\label{eq:abTMA}
\end{equation}
Writing the equation in this way, we have assumed $\alpha=1$, but this is without loss of generality as this normalization can always be achieved by taking roots and replacing $\beta$ by $\beta/\alpha$.

Additionally, we will define the space of admissible functions associated
to a Hermitian metric on which (\ref{eq:abTMA}) is elliptic.
\begin{defn}
Given a Hermitian metric $\omega_{0}$, the set of admissible functions
is defined as 
\[
\mathcal{A}(\omega_{0}):=\{u\in C^4(M)\,|\,\omega_{u}=\omega_{0}+\Box u>0\}.
\]
The tangent space is seen to be 
\begin{equation}
T_{(u,\xi)}\mathcal{A}(\omega)\cong C^4(M).\label{eq:add900}
\end{equation}
\end{defn}

The distinguished case of (\ref{eq:abTMA}) with $\beta=1$ is linear
and has been discussed in the previous section, so we will restrict
our attention to the fully nonlinear case wherein $\beta\in(0,1)\cup(1,\infty)$.
Finally, notice that it suffices to consider the case $\beta\in(0,1)$,
as we may otherwise take roots and achieve similar estimates with
$T^{+}$ and $T^{-}$ swapped. We note that the condition $\beta \neq 1$
will be crucial for our estimates.

We set the following notation conventions to make subsequent computations
easier. First, we define 
\begin{align}
\lambda & :=\omega_{u}^{+}/\omega_{0}^{+},\nonumber \\
\eta & :=\omega_{u}^{-}/\omega_{0}^{-}.\label{eq:add600}
\end{align}
 With these conventions, (\ref{eq:abTMA}) takes the form
\begin{equation}
\lambda^{\beta}=e^{F+\xi}\eta.\label{eq:AlphaTMA2}
\end{equation}
In this section, we obtain uniform estimates for this equation which
will be later used along a continuity path. For convenience and
without loss of generality, we may assume from now on that 
\begin{equation}
\inf_{M}u=0,\label{eq:add610}
\end{equation}
which is equivalent to $u\geq0.$

\subsection{Estimate for the parameter}

One can estimate $\xi$ by a maximum principle argument. This will have
the added benefit of allowing us to avoid needing to track the $\xi$-dependence
of the constants in our estimates.
\begin{lem}
\label{lem:best}Notation as above. Suppose $u\in\mathcal{A}(\omega_{0})$
and $\xi\in\mathbb{R}$ solves Equation \ref{eq:abTMA} for any $\alpha\in\mathbb{R}$,
then $\xi$ is controlled by the sup-norm of $F$. In particular,
\[
|\xi|\leq\|F\|_{\infty}.
\]
\end{lem}

\begin{proof}
At a maximum $x$ of $u$, one has $\i\del\delb u(x)\leq0$ which
implies $\lambda(x)\leq1$ and $\eta(x)\geq1$, turning Equation \ref{eq:AlphaTMA2}
at this point into the inequality 
\[
1\geq e^{F(x)+\xi}.
\]
Therefore, $\xi+F(x)\leq0$, i.e. $\xi\leq-\inf F$. Similarly, at a minimum
$y$, $\i\del\delb u(y)\geq0$ so that Equation \ref{eq:AlphaTMA2}
implies $\xi +F(y)\geq0$, i.e. $\xi \geq-\sup F$.
\end{proof}

\subsection{$C^{0}$-estimate}

We will begin with a Laplace lower bound estimate. 
\begin{lem}
Notation as above. There exists a universal constant $C>0$ depending
only on $F$ so that for any $u\in \mathcal{A}(\omega_0)$,
\begin{equation}
\Delta_{0}u\geq-(1-\beta)C^{\frac{1}{1-\beta}}.\label{eq:0LaplacianBound}
\end{equation}
\end{lem}

\begin{proof}
By Lemma \ref{lem:best} and (\ref{eq:AlphaTMA2}), there exists a
positive constant $C=C(F)$ such that that $\eta\leq C\lambda^{\beta}.$
A direct computation using Young's inequality shows that 
\[
\Delta_{0}u=\lambda-\eta\geq\lambda-C\lambda^{\beta}\geq-(1-\beta)C^{\frac{1}{1-\beta}}.
\]
\end{proof}
We then derive the $L^{1}$-estimate for solutions to our equation.
\begin{thm}
\label{thm:L1apriori}Notations as above. Suppose that $u\in\mathcal{A}(\omega_{0})$
and $\xi \in\mathbb{R}$ solving Equation \ref{eq:abTMA}. Then there
is a constant $C=C(\|F\|_{C^{2}},\omega_{0},\beta)>0$ so that
\[
\|u\|_{L^{1}}\leq C.
\]
\end{thm}

\begin{proof}
We will make use of work of Alesker-Shelukhin \cite[Appendix A]{Alesker2013309}
on the existence of Green's functions for the Chern-Laplacian following
Chu-Tosatti-Weinkove \cite{CTW}. By Theorem \ref{thm:Green}, there
is a non-negative Green's function $G(p,q)$ which is smooth on $M\times M$
away from the diagonal. Furthermore,
\begin{equation}
\int_{M}G(p,q)\Delta_{0}\phi(q)\frac{\omega_{0}^{2}(q)}{2}=\phi(p)-\frac{1}{|M|_{0}}\int_{M}\phi\frac{\omega_{0}^{2}}{2},\forall\phi\in C^{\infty}(M),p\in M.\label{eq:add300}
\end{equation}
Apply (\ref{eq:add300}) to $u$ with $\min_{M}u=0$ and evaluating
at a minimal point $p_{0}$ gives
\[
-\frac{1}{|M|_{0}}\|u\|_{1}=\int_{M}G(p_{0},q)(\Delta_{0}u(q))\frac{\omega_{0}^{2}(q)}{2}.
\]
After a rearrangement and making application of (\ref{eq:0LaplacianBound}),
we find
\[
\|u\|_{1}=|M|_{0}\int_{M}G(p_{0},q)(-\Delta_{0}u(q))\frac{\omega_{0}^{2}(q)}{2}\leq C\int_{M}G(p_{0},q)\frac{\omega_{0}^{2}(q)}{2}\leq C.
\]
Where the last inequality follows from Theorem \ref{thm:Green}. We
have proved the estimate.
\end{proof}
\begin{thm}
\label{thm:oscbyL1}Notation as above. Suppose that $u\in\mathcal{A}(\omega_{0})$
and $\xi\in\mathbb{R}$ solving (\ref{eq:abTMA}). There is a
constant $C=C(\|F\|_{C^{2}},\omega_{0},\beta)>0$ so that 
\[
\sup u\leq C\|u\|_{1}.
\]
This theorem holds true whether the measure in question is $\omega_{0}^{2}/2$
or the normalized probability measure $\mu=\omega_{0}^{2}/2|M|_{0}$.
Therefore, we have $|u|$ uniformly bounded 
\end{thm}

\begin{proof}
This result follows immediately from Theorem \ref{thm:L1apriori}
and Lemma 3.4 of Tosatti-Weinkove \cite{TWBalancedandHerm}, which
employs an iteration method for sub-harmonic functions on surfaces
with Gauduchon metric.
\end{proof}

\subsection{Estimates on diagonal terms of the Hessian}

In this sub-section, we aim to establish the crucial estimate that
bounds $\lambda$ and $\eta$ from both sides. We will make use of
the the following expressions for the linearized operator. For any
smooth function $\phi$ over $M,$ we define
\begin{equation}
L\phi=\beta\frac{\i\del_{+}\delb_{+}\phi}{\omega_{u}^{+}}+\frac{\i\del_{-}\delb_{-}\phi}{\omega_{u}^{-}}=\frac{\beta}{\lambda}\frac{\i\del_{+}\delb_{+}\phi}{\omega_{0}^{+}}+\frac{1}{\eta}\frac{\i\del_{-}\delb_{-}\phi}{\omega_{0}^{-}}.\label{eq:LinearizedAlphaTMA}
\end{equation}
We first state some computational results.
\begin{lem}
\label{lem:Lu}If $u\in\mathcal{A}(\omega_{0})$ solves (\ref{eq:abTMA})
with $\beta\in(0,1)$, then for any $\delta>0$, the following inequality
holds. 

\[
-Lu=\beta\frac{1}{\lambda}-\frac{e^{F+\xi}}{\lambda^{\beta}}+(1-\beta).
\]
\end{lem}

\begin{proof}
Notice that since $\omega_{u}=\omega+\Box u$, by Equation \ref{eq:LinearizedAlphaTMA},
\begin{equation}
Lu=\beta\frac{\omega_{u}^{+}-\omega_{0}^{+}}{\omega_{u}^{+}}+\frac{\omega_{0}^{-}-\omega_{u}^{-}}{\omega_{u}^{-}}=\frac{1}{\eta}-\beta\frac{1}{\lambda}+(\beta-1).\label{eq:Lueqn}
\end{equation}
Combining (\ref{eq:Lueqn}) and (\ref{eq:AlphaTMA2}), we have proved
the claim.
\end{proof}
\begin{lem}
\label{lem:loglambda}If $u\in\mathcal{A}(\omega_{0})$
and $\xi\in\mathbb{R}$ as a pair solves (\ref{eq:abTMA}), then 
\begin{align*}
-L\log\lambda & =\frac{1}{\lambda}(|\del_{+}\log\eta|_{0}^{2}+2\Re(\langle\del_{+}\log\eta,\del_{+}\omega_{0}^{-}\rangle_{0})-\Delta_{0}^{+}F-\frac{\i\del_{-}\delb_{-}\omega_{0}^{+}}{\omega_{0}^{+}\wedge\omega_{0}^{-}})\\
 & +\frac{e^{F+\xi}}{\lambda^{\beta}}(|\del_{-}\log\lambda|_{0}^{2}+2\Re(\langle\del_{-}\log\lambda,\del_{-}\omega_{0}^{+}\rangle_{0})+\frac{\i\del_{-}\delb_{-}\omega_{0}^{+}}{\omega_{0}^{+}\wedge\omega_{0}^{-}}).
\end{align*}
 where $\Delta_{0}^{+}F:=(\i\del_{+}\delb_{+}F\wedge\omega_{0}^{-})/(\omega_{0}^{+}\wedge\omega_{0}^{-})$
and for a section $\mu=\mu_{+}+\mu_{-}\in\Lambda^{1,0}$, the norm
is defined as
\[
|\mu|_{0}^{2}=\frac{\i\mu_{+}\wedge\bar{\mu}_{+}\wedge\omega_{0}^{-}}{\omega_{0}^{+}\wedge\omega_{0}^{-}}+\frac{\i\mu_{-}\wedge\bar{\mu}_{-}\wedge\omega_{0}^{+}}{\omega_{0}^{+}\wedge\omega_{0}^{-}}.
\]
.
\end{lem}

\begin{proof}
Differentiating the logarithm of (\ref{eq:AlphaTMA2}), we get
\begin{equation}
\delb_{+}F=\beta\delb_{+}\log\lambda-\delb_{+}\log\eta,\label{eq:dF}
\end{equation}
\begin{equation}
\i\del_{+}\delb_{+}F=\beta\i\del_{+}\delb_{+}\log\lambda-\i\del_{+}\delb_{+}\log\eta.\label{eq:ddcF}
\end{equation}
For future use, we also list some direct consequence of the pluriclosed
condition. By (\ref{eq:add600}), $\omega_{u}=\lambda\omega_{0}^{+}+\eta\omega_{0}^{-}$
. Since both $\omega_{0}$ and $\omega_{u}$ are pluriclosed, we
have
\begin{equation}
\i\del\delb\omega_{0}=\i\del_{-}\delb_{-}\omega_{0}^{+}+\i\del_{+}\delb_{+}\omega_{0}^{-}=0;\label{eq:add511}
\end{equation}
\begin{align}
0= & \i\del\delb(\lambda\omega_{0}^{+}+\eta\omega_{0}^{-})\label{eq:pluriclosedcond}\\
= & \i\del_{-}\delb_{-}\lambda\wedge\omega_{0}^{+}+\i\del_{+}\delb_{+}\eta\wedge\omega_{0}^{-}\nonumber \\
 & +2\Re(\i\del_{-}\lambda\wedge\delb_{-}\omega_{0}^{+})+2\Re(\i\del_{+}\eta\wedge\delb_{+}\omega_{0}^{-})\nonumber \\
 & +\i\lambda\del_{-}\delb_{-}\omega_{0}^{+}+\i\eta\del_{+}\delb_{+}\omega_{0}^{-}.\nonumber 
\end{align}
Now we compute $L\log\lambda$ using (\ref{eq:add600}) and (\ref{eq:LinearizedAlphaTMA}).
\begin{align*}
L\log\lambda= & \beta\frac{\i\del_{+}\delb_{+}\log\lambda\wedge\omega_{0}^{-}}{\lambda\omega_{0}^{+}\wedge\omega_{0}^{-}}+\frac{\i\del_{-}\delb_{-}\log\lambda\wedge\omega_{0}^{+}}{\eta\omega_{0}^{+}\wedge\omega_{0}^{-}}\\
= & \beta\frac{\i\del_{+}\delb_{+}\log\lambda\wedge\omega_{0}^{-}}{\lambda\omega_{0}^{+}\wedge\omega_{0}^{-}}+\frac{\i\del_{-}\delb_{-}\lambda\wedge\omega_{0}^{+}}{\lambda\eta\omega_{0}^{+}\wedge\omega_{0}^{-}}-\frac{1}{\lambda^{2}\eta}\frac{\i\del_{-}\lambda\wedge\delb_{-}\lambda\wedge\omega_{0}^{+}}{\omega_{0}^{+}\wedge\omega_{0}^{-}}\\
\end{align*}
We can then apply  (\ref{eq:add511}) and (\ref{eq:pluriclosedcond}).
\begin{align*}
L\log \lambda = & \beta\frac{\i\del_{+}\delb_{+}\log\lambda\wedge\omega^{-}}{\lambda\omega_{0}^{+}\wedge\omega_{0}^{-}}-\frac{1}{\lambda^{2}\eta}\frac{\i\del_{-}\lambda\wedge\delb_{-}\lambda\wedge\omega_{0}^{+}}{\omega_{0}^{+}\wedge\omega_{0}^{-}}\\
 & -\frac{\i\del_{+}\delb_{+}\eta\wedge\omega_{0}^{-}+2\Re(\i\del_{-}\lambda\wedge\delb_{-}\omega_{0}^{+})+2\Re(\i\del_{+}\eta\wedge\delb_{+}\omega_{0}^{-})}{\lambda\eta\omega_{0}^{+}\wedge\omega_{0}^{-}}\\
 & -\frac{\i\lambda\del_{-}\delb_{-}\omega_{0}^{+}+\i\eta\del_{+}\delb_{+}\omega_{0}^{-}}{\lambda\eta\omega_{0}^{+}\wedge\omega_{0}^{-}}\\
 \end{align*}
 Finally, we can use (\ref{eq:dF}) and (\ref{eq:ddcF}) to obtain.
 \begin{align*}
L\log \lambda = & \beta\frac{\i\del_{+}\delb_{+}\log\lambda\wedge\omega_{0}^{-}}{\lambda\omega_{0}^{+}\wedge\omega_{0}^{-}}-\frac{\i\del_{+}\delb_{+}\log\eta\wedge\omega_{0}^{-}}{\lambda\omega_{0}^{+}\wedge\omega_{0}^{-}}\\
 & -\frac{1}{\lambda\eta^{2}}\frac{\i\del_{+}\eta\wedge\delb_{+}\eta\wedge\omega_{0}^{-}}{\omega_{0}^{+}\wedge\omega_{0}^{-}}-\frac{1}{\lambda^{2}\eta}\frac{\i\del_{-}\lambda\wedge\delb_{-}\lambda\wedge\omega_{0}^{+}}{\omega_{0}^{+}\wedge\omega_{0}^{-}}\\
 & -\frac{2\Re(\i\del_{-}\lambda\wedge\delb_{-}\omega_{0}^{+})+2\Re(\i\del_{+}\eta\wedge\delb_{+}\omega_{0}^{-})}{\lambda\eta\omega_{0}^{+}\wedge\omega_{0}^{-}}\\
 & -\frac{\i\lambda\del_{-}\delb_{-}\omega_{0}^{+}+\i\eta\del_{+}\delb_{+}\omega_{0}^{-}}{\lambda\eta\omega_{0}^{+}\wedge\omega_{0}^{-}}\\
 \end{align*}
 This can be simplified to 
 \begin{align*}
L\log \lambda = & \frac{1}{\lambda}(\Delta_{0}^{+}F-\frac{\i\del_{+}\delb_{+}\omega_{0}^{-}}{\omega_{0}^{+}\wedge\omega_{0}^{-}}-2\Re(\langle\del_{+}\log\eta,\del_{+}\omega_{0}^{-}\rangle_{0})-|\del_{+}\log\eta|_{0}^{2})\\
 & +\frac{1}{\eta}(-\frac{\i\del_{-}\delb_{-}\omega_{0}^{+}}{\omega_{0}^{+}\wedge\omega_{0}^{-}}-|\del_{-}\log\lambda|_{0}^{2}-2\Re(\langle\del_{-}\log\lambda,\del_{-}\omega_{0}^{+}\rangle_{0})).
\end{align*}
This can be simplified further by using (\ref{eq:abTMA}) and (\ref{eq:add511}), finishing the proof.
\end{proof}
\begin{lem}
\label{lem:logPhi}Notations as above. Consider $\Phi=\log\lambda+\psi(u)$
for some smooth test function $\psi$. Then at any critical point
of $\Phi$, the following holds.
\begin{align*}
 & -L\Phi=\\
 & \frac{1}{\lambda}[|\del_{+}\log\eta|_{0}^{2}+2\Re(\langle\del_{+}\log\eta,\del_{+}\omega_{0}^{-}\rangle_{0})-\beta\frac{\psi''}{(\psi')^{2}}|\del_{+}\log\lambda|_{0}^{2}+\beta\psi'-\frac{\i\del_{-}\delb_{-}\omega_{0}^{+}}{\omega_{0}^{+}\wedge\omega_{0}^{-}}-\Delta_{0}^{+}F]\\
 & +\frac{e^{F+\xi}}{\lambda^{\beta}}[(1-\frac{\psi''}{(\psi')^{2}})|\del_{-}\log\lambda|_{0}^{2}+2\Re(\langle\del_{-}\log\lambda,\del_{-}\omega_{0}^{+}\rangle_{0})-\psi'+\frac{\i\del_{-}\delb_{-}\omega_{0}^{+}}{\omega_{0}^{+}\wedge\omega_{0}^{-}}].\\
 & +\psi'(1-\beta)
\end{align*}
\end{lem}

\begin{proof}
The critical point condition is 
\begin{equation}
0=\frac{d\lambda}{\lambda}+\psi'du.\label{eq:critpointcond}
\end{equation}
We may compute the following, while applying (\ref{eq:critpointcond}):
\begin{align}
L\psi(u)= & \psi'Lu+\psi''[\frac{\beta}{\lambda}|\del_{+}u|_{0}^{2}+\frac{1}{\eta}|\del_{-}u|_{0}^{2}]\nonumber \\
= & \psi'Lu+\frac{\psi''}{(\psi')^{2}}[\frac{\beta}{\lambda}|\del_{+}\log\lambda|_{0}^{2}+\frac{e^{F+\xi}}{\lambda^{\beta}}|\del_{-}\log\lambda|_{0}^{2}]\nonumber \\
= & \frac{\beta}{\lambda}(\frac{\psi''}{(\psi')^{2}}|\del_{+}\log\lambda|_{0}^{2}-\psi')+\frac{e^{F+\xi}}{\lambda^{\beta}}(\frac{\psi''}{(\psi')^{2}}|\del_{-}\log\lambda|_{0}^{2}+\psi')-\psi'(1-\beta)\label{eq:add602}
\end{align}
 We then combine (\ref{eq:add602}) and Lemma \ref{lem:loglambda}
to prove our claim.
\end{proof}
Finally, we state first part of our $C^{2}$ estimate:
\begin{thm}
\label{thm:lowerbound}Notation as before. For any $u\in\mathcal{A}(\omega_{0})$
and $\xi\in\mathbb{R}$ solving (\ref{eq:abTMA}) for $\beta\in(0,1)$,
there exists $C>0$ depending only on $\beta$, $\|F\|_{C^{2}}$,
and $\omega_{0}$ s.t.
\[
\omega_{u}^{+}\geq C\omega_{0}^{+}.
\]
\end{thm}

\begin{proof}
We use $C$, $C_{1}$ to denote a positive constant that depends only
on $\beta,$ $\|F\|_{C^{2}}$, and $\omega_{0}$, which may change
from line to line unless otherwise mentioned. We take $\psi(x)=Ax$
in Lemma \ref{lem:logPhi}, whereas $A>0$ is a constant to be determined
later. By the Cauchy-Schwarz and arithmetic-geometric mean inequalities,
we have
\begin{align}
2\Re(\langle\del_{+}\log\eta,\del_{+}\omega_{0}^{-}\rangle_{0}) & \geq-\frac{1}{2}|\del_{+}\log\eta|_{0}^{2}-C,\label{eq:700}\\
2\Re(\langle\del_{-}\log\lambda,\del_{-}\omega_{0}^{+}\rangle_{0}) & \geq-\frac{1}{2}|\del_{-}\log\lambda|_{0}^{2}-C.\label{eq:add601}
\end{align}
 Therefore, for any $\epsilon>0$, at point $p\in M$ for which $\Phi(p)=\min\Phi$,
by Lemma \ref{lem:best}, Lemma \ref{lem:logPhi}, (\ref{eq:700})
and (\ref{eq:add601}), we have:

\begin{align}
0\geq-L\Phi\geq & \frac{1}{\lambda}(\beta A-C)-\frac{C}{\lambda^{\beta}}(A+1)+A(1-\beta).\nonumber \\
\geq & \frac{1}{\lambda}[\beta A-C-\beta C\epsilon(A+1)]+(1-\beta)A-C\epsilon^{-\frac{\beta}{1-\beta}}(A+1)\nonumber \\
> & \frac{1}{\lambda}[\beta A(1-\epsilon C)-C-\beta C\epsilon]-C\epsilon^{-\frac{\beta}{1-\beta}}(A+1),\label{eq:add603}
\end{align}
where the second line follows from Young's inequality. We first fix
$C>0$, then we pick $\epsilon=\frac{1}{2C}$. (\ref{lem:best}) then
leads to 

\begin{equation}
0\geq[C_{1}A-C_{2}]-C_{3}(A+1)\lambda(p),\label{eq:add604}
\end{equation}
where $C_{i}=C_{i}(\beta,\|F\|_{C^{2}},\omega_{0})>0.$ Then, we
pick $A=2C_{2}/C_{1}>0$ in (\ref{eq:add604}) to obtain

\[
\lambda(p)\geq C>0.
\]
Finally, by Theorem \ref{thm:oscbyL1}, and the choice of $p$ and
$\Phi$, we obtain our estimate.
\end{proof}
Next, we apply Lemma \ref{lem:logPhi} with a different test function
to obtain the upper bound. 
\begin{thm}
\label{thm:upperbound}Supposing $u\in\mathcal{A}(\omega_{0})$
and $\xi\in\mathbb{R}$ as a pair is a solution of (\ref{eq:abTMA})
for $\beta\in(0,1)$, then there exists a constant $C=C(\|F\|_{C^{2}},\omega_{0},\beta)>0$
such that
\[
\omega_{u}^{+}\leq C\omega_{0}^{+}.
\]
\end{thm}

\begin{proof}
We consider a test function $\Phi=\log\lambda+\psi(u)$ for some smooth
test function $\psi$, which is to be deterimined later. We also consider
point $p\in M$ such that $\Phi(p)=\max\Phi.$ 

First, we use the Cauchy-Schwarz and Young's inequalities to find
that for any $\delta>0$, the following hold:
\begin{align}
 & |\del_{+}\log\eta|^{2}+2\Re(\langle\del_{+}\log\eta,\del_{+}\omega_{0}^{-}\rangle_{0})\label{eq:add606}\\
= & |\del_{+}(F-\beta\log\lambda)|_{0}^{2}+2\Re(\langle\del_{+}(F-\beta\log\lambda),\del_{+}\omega_{0}^{-}\rangle_{0})\\
\leq & \beta^{2}|\del_{+}\log\lambda|_{0}^{2}+\beta C|\del_{+}(\log\lambda)|_{0}+C\nonumber \\
\leq & (\beta^{2}+\delta\beta)|\del_{+}\log\lambda|_{0}^{2}+C(1+\frac{1}{\delta});\nonumber 
\end{align}
\begin{align*}
2\Re(\langle\del_{-}\log\lambda,\del_{-}\omega_{0}^{+}\rangle_{0}) & \leq\delta|\del_{-}\log\lambda|_{0}^{2}+\frac{C}{\delta},
\end{align*}
We apply Lemma \ref{lem:logPhi} and (\ref{eq:add606}) to obtain
the following estimate at point $p$: 
\begin{align}
0\leq & \frac{1}{\lambda}(\beta^{2}+\delta\beta-\frac{\psi''}{(\psi')^{2}})|\del_{+}\log\lambda|_{0}^{2}+\beta\psi'+C(1+\frac{1}{\delta}))]\nonumber \\
 & +\frac{C}{\lambda^{\beta}}[(1+\delta-\frac{\psi''}{(\psi')^{2}})|\del_{-}\log\lambda|_{0}^{2}-\psi'+C(1+\frac{1}{\delta})]\nonumber \\
 & +\psi'(1-\beta).\label{eq:666}
\end{align}
We now pick $\psi(x)=\tau x-\log(x+1)$ where $\tau=\frac{1}{4}(1+\mathrm{osc}u)^{-1}$.
Then by Theorem \ref{thm:oscbyL1} and (\ref{eq:add610}), we get
\begin{align}
\psi'(u)= & \tau-\frac{1}{1+u}\in[\tau-1,-3\tau]\label{eq:668}\\
\psi''(u)= & \frac{1}{(1+u)^{2}}\geq16\tau^{2}>0,\nonumber 
\end{align}
and
\[
\frac{\psi''(u)}{(\psi'(u))^{2}}=\frac{1}{(\tau(1+u)-1)^{2}}\geq\frac{16}{9}.
\]
Thus, if $\delta=1/2,$ we have 
\begin{equation}
\beta^{2}+\delta\beta-\frac{\psi''}{(\psi')^{2}}<1+\delta-\frac{\psi''}{(\psi')^{2}}<\frac{3}{2}-\frac{16}{9}<0.\label{eq:667}
\end{equation}
\[
\]
Now we combine (\ref{eq:666}), (\ref{eq:668}) and (\ref{eq:667})
to get bounded $C_{1},$ $C_{2}$, $C_{3},$ where $C_{1}=3\tau>0$
and $C_{2}>0,$ such that
\begin{align*}
0\leq & C_{3}\lambda(p)^{-1}+C_{1}\lambda(p)^{-\beta}-C_{2},
\end{align*}
which, by Young's inequality, leads to
\begin{align*}
0\leq & C_{3}+C_{1}(C_{\epsilon}+\epsilon\lambda(p))-C_{2}\lambda(p),
\end{align*}
where $C_{\epsilon}>0$ depending on $\epsilon.$ Now we may choose
$\epsilon=\frac{C_{2}}{2C_{1}}$ to get 
\begin{equation}
\lambda(p)\leq C<\infty.\label{eq:add669}
\end{equation}

Finally, by Theorem \ref{thm:oscbyL1}, and the choice of $p$ and
$\Phi$, we obtain our estimate.
\end{proof}
Theorems \ref{thm:lowerbound} \& \ref{thm:upperbound} imply that
the metric is uniformly equivalent to the background metric.
\begin{cor}
\label{cor:fullmetricbound}As an immediate consequence of Theorems
\ref{thm:lowerbound} \& \ref{thm:upperbound} with $\beta\in(0,1)$
and $u$, $b$ as noted in those theorems,
\begin{align}
C^{-1}\omega_{0}\leq & \omega_{u}\leq C\omega_{0}.\label{eq:fullmetricbound}
\end{align}
In particular, the linearized operator $L$ defined in (\ref{thm:alpha})
is uniformly elliptic.
\end{cor}

\begin{proof}
This is just a direct consequence of Theorems \ref{thm:lowerbound}
\& \ref{thm:upperbound}, Lemma \ref{lem:best}, and our equation
\ref{eq:AlphaTMA2}. 
\end{proof}

\subsection{Full $C^{2}$ estimate}

In this subsection, we estimate the mixed second derivatives. Unlike
the situation in the usual Monge-Amp\`ere equations, where there is
direct control of off-diagonal terms of Hessian due to the PDE, these
mixed derivatives are not appearing in our geometric equations directly. 

First, we list some regularity results that may be obtained already.
\begin{prop}
\label{prop:W2p}Notation as above. For any $p>1$ and $\epsilon>0,$there
exists a univeral constant $C$ and depends only on $(F,\omega_{0},p,\epsilon)$
such that $\|\nabla^{2}u\|_{L^{p}}\leq C,\ \|\nabla u\|_{C^{\epsilon}}\leq C.$
\end{prop}

\begin{proof}
By Theorems \ref{thm:lowerbound} \& \ref{thm:upperbound}, we have
obtained uniform lower and upper bound of $\Delta u$, which implies
the Hessian $L^{p}$ estimate. We can then apply the standard Soblev
inequality to obtain the $C^{1,\epsilon}$ estimate of $u.$ 
\end{proof}
In general, a Laplacian bound does not imply a uniform bound of the
full Hessian in non-linear PDE theory. Proposition \ref{prop:W2p}
is the optimal estimate. 

We shall run the $C^{2}$ estimate one more time, which will greatly
simplified with existing estimates. In order to proceed, we use Lemma
\ref{lem:localcoordinate} to work locally in an open neighborhood $U$
of $M$, where there exists local holomorphic function $z$ and $w$
such that in $U,$ $T^{+}=\mathrm{span\{\frac{\partial}{\partial z}\}}$
and $T^{-}=\mathrm{span}\{\frac{\partial}{\partial w}\}$. We also
write
\[
\omega_{0}=\sqrt{-1}(gdz \wedge d\bar{z}+hdw\wedge d\bar{w}).
\]

\begin{rem}
\label{rem:local data}Since $M$ is compact, by a covering argument,
we may assume $g,h,g^{-1},h^{-1}$, and their derivatives are universally
bounded. 
\end{rem}

We continue to express the linearized operator $L$ in (\ref{eq:AlphaTMA2})
locally 
\begin{equation}
L\phi=\beta\frac{\i\del_{+}\delb_{+}\phi}{\omega_{u}^{+}}+\frac{\i\del_{-}\delb_{-}\phi}{\omega_{u}^{-}}=\frac{\beta}{\lambda g}\phi_{z\bar{z}}+\frac{1}{\eta h}\phi_{w\bar{w}}.\label{eq:add86}
\end{equation}
By Corollary \ref{cor:fullmetricbound}, $L$ is uniformly elliptic.

Let $\{x^{i}\}_{i=1}^{4}$be $\Re(z),\Im(z),\Re(w),\Im(w)$, respectively.
We use $\tilde{g}$ to denote the standard Euclidean metric in $U$
with coordinates $\{x^{i}\}.$ Let $\delta=\sum a^{i}\frac{\partial}{\partial x^{i}}$
be a local vector field in $U$, where $a^{i}\in\mathbb{R}$ with
$\sum a_{i}^{2}\leq1$. For simplicity, for any smooth function f,
we write $\delta f$ as $f_{\delta}$. It is clear that by our set-up
\begin{equation}
|f_{\delta}|^{2}\leq|\tilde{\nabla}f|^{2}:=\sum|\frac{\partial f}{\partial x^{i}}|^{2}.\label{eq:add41}
\end{equation}
For future use, we also define 
\begin{equation}
\tilde{\Delta}u:=\sum(\frac{\partial}{\partial x^{i}})^{2}u=\i(u_{z\bar{z}}+u_{w\bar{w}}).\label{eq:add87}
\end{equation}

\begin{rem}
It is important to note that in the setting above, we have $f_{\delta z}=f_{z\delta},$
$f_{\delta w}=f_{w\delta},$$f_{\delta\bar{z}}=f_{\bar{z}\delta},$$f_{\delta\bar{w}}=f_{\bar{w}\delta}$.
\end{rem}

\begin{lem}
\label{lem:subsolution}For $\delta u$, locally we have univeral
constants $C_{1}$, $C_{2}$ and $C_{3}$ that is independent of $u$
such that
\begin{align}
L(u_{\delta\delta}) & \geq C_{1};\label{eq:add80}\\
L(C_{2}\Delta u-u_{\delta\delta}) & \geq C_{3.}\label{eq:add81}
\end{align}
\end{lem}

\begin{proof}
We apply $\delta$ to our PDE (\ref{eq:AlphaTMA2}) twice to obtain
\begin{align}
\beta\frac{\lambda_{\delta}}{\lambda}-\frac{\eta_{\delta}}{\eta} & =F_{\delta},\label{eq:add71}\\
\beta\frac{\lambda_{\delta\delta}}{\lambda}-\frac{\eta_{\delta\delta}}{\eta} & -\beta(\frac{\lambda_{\delta}}{\lambda})^{2}+(\frac{\eta_{\delta}}{\eta})^{2}=F_{\delta\delta},\label{eq:add72}
\end{align}
which lead to 
\begin{align}
\beta\frac{\delta^{2}\lambda}{\lambda}-\frac{\delta^{2}\eta}{\eta} & =\delta^{2}F+\beta(\frac{\delta\lambda}{\lambda})^{2}-(\beta\frac{\delta\lambda}{\lambda}-\delta F)^{2}.\label{eq:add73}\\
 & =F_{\delta\delta}+F_{\delta}^{2}+(\beta-\beta^{2})(\frac{\lambda_{\delta}}{\lambda})^{2}+2\beta\frac{\lambda_{\delta}}{\lambda}F_{\delta}.\nonumber 
\end{align}
On the other hand, using Theorems \ref{thm:lowerbound} \& \ref{thm:upperbound}
and Remark \ref{rem:local data}, we have 
\begin{equation}
\beta\frac{\delta^{2}\lambda}{\lambda}-\frac{\delta^{2}\eta}{\eta}=\frac{\beta}{\lambda}(1+\frac{u_{z\bar{z}}}{g})_{\delta\delta}-\frac{1}{\eta}(1+\frac{u_{w\bar{w}}}{g}))_{\delta\delta}=Lu+K_{1}(\frac{\lambda_{\delta}}{\lambda})+K_{2}\frac{\eta_{\delta}}{\eta}\label{eq:add74}
\end{equation}
where $K_{i},$$i=1,2,3$ are some universally bounded functions.
Note that for any $\epsilon>0$
\begin{equation}
|\frac{\lambda_{\delta}}{\lambda}F_{\delta}|\leq\epsilon|\lambda_{\delta}|^{2}+\frac{1}{2\epsilon}|F_{\delta}|^{2}\label{eq:add78}
\end{equation}
 and we use (\ref{eq:add73}) and (\ref{eq:add74}) to conclude
\begin{equation}
C_{5}+C_{6}|\frac{\lambda_{\delta}}{\lambda}|^{2}\leq L(u_{\delta\delta})\leq C_{7}+C_{8}|\frac{\lambda_{\delta}}{\lambda}|^{2},\label{eq:add75}
\end{equation}
where $C_{i}$ are universally bounded. In particular, since $\beta<1,$
we may pick $\epsilon=C_{4}>0$ in (\ref{eq:add78}) small enough
to ensure that $C_{6}>0$ and $C_{8}>0$. Therefore, we have proved
(\ref{eq:add80}), the first part of our claim.

To prove the second half of our claim, we consider $\tilde{\Delta}u$.
Apply \ref{eq:add75} to $\delta=\frac{\partial}{\partial x^{i}}$ repeatedly
and sum up resulting inequalities, we obtain 
\begin{equation}
C_{9}+C_{10}|\frac{\tilde{\nabla}\lambda}{\lambda}|^{2}\leq L(\tilde{\Delta}u)\leq C_{11}+C_{12}|\frac{\tilde{\nabla}\lambda}{\lambda}|^{2}\label{eq:add79}
\end{equation}
Again, we may have $C_{10}>0$ and $C_{12}>0$.
\[
\]

Now notice that by (\ref{eq:add41})
\begin{equation}
|\frac{\lambda_{\delta}}{\lambda}|^{2}\leq|\frac{\tilde{\nabla}\lambda}{\lambda}|^{2}.\label{eq:add83}
\end{equation}
We can then use (\ref{eq:add75}) and (\ref{eq:add79}) to get, for
any $C_{2}$
\[
L(C_{2}\tilde{\Delta}u-u_{\delta\delta})\geq C_{2}C_{10}(|\frac{\tilde{\nabla}\lambda}{\lambda}|^{2}+|\frac{\tilde{\nabla}\eta}{\eta}|^{2})-C_{8}(|\frac{\lambda_{\delta}}{\lambda}|^{2}+|\frac{\eta_{\delta}}{\eta}|^{2})+C_{11}.
\]
Finally, we may pick $C_{2}=C_{8}/C_{10}$ and use can apply (\ref{eq:add83})
to get (\ref{eq:add80}).
\end{proof}
Finally, we may establish the full $C^{2}$ estimate.
\begin{thm}
Notation as above. There exists a universal constant $C$ that depends
only on $\beta,F$ and $\omega_{0}$ such that we have the following
Hessian bound
\[
|\nabla^{2}u|\leq C.
\]
\end{thm}

\begin{proof}
We work in the open set $U$ defined as above. For $u_{\delta\delta}$
and $C_{2}\tilde{\Delta}u-u_{\delta\delta}$ defined as in Lemma \ref{lem:subsolution}.
By Corollary \ref{cor:fullmetricbound}, the linearlized operator
$L$ in (\ref{eq:add86}) is linear and uniformly elliptic. Therefore,
we may apply Proposition \ref{prop:W2p} and the local maximum principle
for subsolutions (Theorem 4.8 part 2 of \cite{gt}) to conclude that $u_{\delta\delta}\leq C$
and $C_{2}\tilde{\Delta}u-u_{\delta\delta}\leq C$ . Also note that
by (\ref{eq:add87}) and Corollary \ref{cor:fullmetricbound},$\tilde{\Delta}u\geq C.$
We then conclude that 
\[
|u_{\delta\delta}|\leq C.
\]
Finally, by considering $\delta=\frac{1}{\sqrt{2}}(\frac{\partial}{\partial x^{i}}+\frac{\partial}{\partial x^{i}}),$
we may get the desired bound for all mixed second derivatives. We
have finished the proof. 
\end{proof}

\subsection{Higher regularity}

Our previous method may be used to consider the higher regularity
estimate. However, an alternative is to apply the following Evans-Krylov
theorem for twisted type operators due to Collins \cite{collins:twistedtype}.
See also Streets-Warren \cite{sw:evanskrylov} for a related result.
\begin{defn}
On a Riemannian manifold, an elliptic operator $\Psi=\Psi_{\cup}+\Psi_{\cap}$
is said to be of \emph{twisted type} if 
\begin{enumerate}
\item $\Psi_{\cup}$ is uniformly elliptic and convex, and
\item $\Psi_{\cap}$ is degenerate elliptic and concave.
\end{enumerate}
This definition is not as general as that found in Collins' paper,
but will suffice for our purposes.
\end{defn}

\begin{lem}
Notation as above. There exists $\epsilon>0$ sufficiently small so
that
\[
\Psi=(\beta\log\frac{\omega_{u}^{+}}{\omega_{0}^{+}}-\epsilon\frac{\omega_{u}^{-}}{\omega_{0}^{-}})+(\epsilon\frac{\omega_{u}^{-}}{\omega_{0}^{-}}-\log\frac{\omega_{u}^{-}}{\omega_{0}^{-}})
\]
is elliptic and, when split as indicated by the parentheses, is of
twisted type.
\end{lem}

\begin{proof}
Letting $\Psi_{\cup}$ and $\Psi_{\cap}$ be defined as follows. 
\[
\Psi_{\cup}=\beta\log\frac{\omega_{u}^{+}}{\omega_{0}^{+}}-\epsilon\frac{\omega_{u}^{-}}{\omega_{0}^{-}},\quad\Psi_{\cap}=\epsilon\frac{\omega_{u}^{-}}{\omega_{0}^{-}}-\log\frac{\omega_{u}^{-}}{\omega_{0}^{-}}
\]
Notice that corresponding linearized operators are 
\[
\delta\Psi_{\cup}(\delta u)=\frac{\beta}{\lambda}\frac{\i\del_{+}\delb_{+}\delta u}{\omega_{0}^{+}}+\epsilon\frac{\i\del_{-}\delb_{-}\delta u}{\omega_{0}^{-}},\quad\delta\Psi_{\cap}(\delta u)=(\frac{1}{\eta}-\epsilon)\frac{\i\del_{-}\delb_{-}\delta u}{\omega_{0}^{-}}.
\]
By Corollary \ref{cor:fullmetricbound}, $\Psi_{\cup}$ is uniformly
elliptic and $\epsilon>0$ can be chosen sufficiently small so that
$\Psi_{\cap}$ is degenerate elliptic. Then since 
\[
\delta (\frac{1}{\lambda})=-\frac{1}{\lambda^{2}}\frac{\i\del_{+}\delb_{+}\delta u}{\omega_{0}^{+}},\quad \delta(\frac{1}{\eta})=\frac{1}{\eta^{2}}\frac{\i\del_{-}\delb_{-}\delta u}{\omega_{0}^{-}},
\]
it is the case that $\Psi_{\cup}$ and $\Psi_{\cap}$ have the appropriate
concavity properties.
\end{proof}
\begin{thm}
\label{thm:collins}\cite[Theorem 3.2]{collins:twistedtype} Suppose
that $u\in C^{\infty}(B_2)$ on $B_{2}\subset\mathbb{R}^{n}$where $F=F_{\cup}+F_{\cap}$.
Let $\mathcal{U}=D^{2}u(\bar{B}_{1})$ and let $\mathcal{V}\supset\mathcal{U}$be
an open and convex set. Suppose that $F_{\cup}$ is uniformly elliptic,
convex and $C^{2}$ on $\mathcal{V}$, and $F_{\cap}$ is $C^{2}$
on $\mathcal{V}$. Assume furthermore that $F_{\cap}$ is degenerate
elliptic and concave on $\mathcal{U}$. Then, for every $\gamma\in(0,1)$
we have an estimate 
\[
\|D^{2}u\|_{C^{\gamma}(B_{\frac{1}{2}})}\leq C(n,\lambda,\Lambda,\gamma,\Gamma,F,\|D^{2}u\|_{\infty})
\]
where 
\[
\Gamma=\mathrm{osc}_{B_{1}}(-F_{\cup}(D^{2}u))
\]
depend only on $\Lambda$ and $\|D^{2}u\|_{\infty}$.
\end{thm}
This theorem in hand, we are able to bootstrap to obtain estimates for $u\in C^k(M)$ for all $k$.

\section{Continuity Method}

In this section, we run the continuity method to solve (\ref{eq:add0})
when $\beta<1$ and $\alpha=1$. As $u\equiv0$ and $\xi =0$ is a solution for $F\equiv 0$,
we seek to apply the continuity method to the path $F_{t}=tF$ with
$F\in C^{\infty}(M)$ for $t\in[0,1]$. We consider the following
PDE
\begin{equation}
\begin{cases}
(\frac{\omega_{u_{t}}^{+}}{\omega_{0}^{+}})^{\beta}=e^{F_{t}+\xi_{t}}(\frac{\omega_{u_{t}}^{-}}{\omega_{0}^{-}})\\
\omega_{u}=\omega_{0}+\Box u>0
\end{cases}.\label{eq:new}
\end{equation}
Let
\begin{equation}
S=\{t\in[0,1]\:|\:\exists(u_{t},\xi_{t})\in\mathcal{\mathcal{A}}(\omega_{0})\times\mathbb{R}\text{ solving (\ref{eq:new})}\}.\label{eq:setS}
\end{equation}
Since $S$ is non-empty, it is sufficient to show that
$S$ is both open and closed in $[0,1]$.

\subsection{Openness}

The proof of openness follows a standard Inverse Function Theorem
argument.
\begin{thm}
On $(M^{2},I)$ a compact, complex surface with Hermitian metric $\omega_{0}$
and $\beta\in\mathbb{R}$, the following map 
\begin{align*}
\Psi:\mathcal{A}(\omega_{0})\times\mathbb{R} & \to C^{\gamma}(M)\\
(u,\xi) & \mapsto\beta\log\frac{\omega_{u}^{+}}{\omega_{0}^{+}}-\log\frac{\omega_{u}^{-}}{\omega_{0}^{-}}-\xi
\end{align*}
 is locally invertible.
\end{thm}

\begin{proof}
We compute the linearized operator at a solution $(u,\xi)$. We use
Theorem \ref{thm:Gauduchonmetric} to obtain a Gauduchon factor for
$\frac{1}{\beta}\omega_{u}^{+}+\omega_{u}^{-}$ and denote it by
$e^{f}$ , such that $\tilde{\omega}=e^{f}(\frac{1}{\beta}\omega_{u}^{+}+\omega_{u}^{-})$
is pluriclosed. Let $(\delta u,\delta \xi)\in C^{4}(M)\times\mathbb{R}\cong T_{(u,\xi)}(\mathcal{A}(\omega)\times\mathbb{R})$.
Define the tangent map
\[
\delta\Psi|{}_{(u,\xi)}(\delta u,\delta \xi)=e^{f}\Delta_{\tilde{\omega}}\delta u-\delta \xi.
\]
By the Inverse Function Theorem for Banach Spaces \cite[Theorem 17.6]{gt}, it is sufficient
to show that $\delta\Psi|_{(u,\xi)}:C^{\gamma}(M)\times\mathbb{R}\to T_{\Psi(u,\xi)}C^{\gamma}(M)$
is an isomorphism. 

We start with injectivity. Suppose that $\delta\Psi|_{(u,\xi)}(\delta u,\delta \xi)=0$.
Then, by Theorem \ref{thm:19}, the Chern-Poisson equation is uniquely
solvable if and only if 
\[
0=\int_{M} (\delta \xi) e^{-f}\tilde{\omega}^{2},
\]
which shows that $\delta \xi=0$. Therefore, $\Delta_{\tilde{\omega}}\delta u=0.$
By the maximum principle and (\ref{eq:new}), $\delta u=0$. Thus,
the kernel of $\delta\Psi|_{(u,\xi)}$ is trivial.

On the other hand, for any $\delta F\in C^{\gamma}(M)$, we use Theorem
\ref{thm:19} to obtain $(\delta u,\delta \xi)\in T_{(u,\xi)}\mathcal{A}(\omega)$
to satisfy the following
\begin{equation}
\delta \xi=-\frac{\int_{M}\delta Fe^{f}\omega_{u}^{2}}{\int_{M}e^{f}\omega_{u}^{2}},\quad\Delta_{\tilde{\omega}}\delta u=e^{-f}(\delta F+\delta \xi),\label{eq:sujectivity}
\end{equation}
which implies surjectivity. 
\end{proof}

\subsection{Closedness}

Finally, we prove the closedness of the set $S,$ which completes the proof
of Theorem \ref{thm:main} in the case $0<\beta<1.$ The closedness
of $S$ is a direct consequence of \emph{a priori} estimates: Theorems
\ref{thm:L1apriori} and \ref{thm:oscbyL1}, Corollary \ref{cor:fullmetricbound},
and Theorem \ref{thm:collins}. Considering Theorem \ref{t:TMAsolved}
which is the case for $\beta=1$ and using a symmetry argument to
treat the $\beta>1$ case, we have therefore established Theorem
\ref{thm:main}.

\bibliographystyle{abbrvnat}
%\bibliography{mybib}

\begin{thebibliography}{47}
\providecommand{\natexlab}[1]{#1}
\providecommand{\url}[1]{\texttt{#1}}
\expandafter\ifx\csname urlstyle\endcsname\relax
  \providecommand{\doi}[1]{doi: #1}\else
  \providecommand{\doi}{doi: \begingroup \urlstyle{rm}\Url}\fi

\bibitem[Alesker and Shelukhin(2013)]{Alesker2013309}
S.~Alesker and E.~Shelukhin.
\newblock On a uniform estimate for the quaternionic {C}alabi problem.
\newblock \emph{Israel J. of Math.}, 197\penalty0 (1):\penalty0 309 – 327, 2013.
\newblock \doi{10.1007/s11856-013-0003-1}.

\bibitem[Angella(2015)]{angella2015bottchern}
D.~Angella.
\newblock On the {B}ott-{C}hern and {A}eppli cohomology, 2015.

\bibitem[Angella and Tosatti(2021)]{Angella2021LeafwiseFF}
D.~Angella and V.~Tosatti.
\newblock Leafwise flat forms on {I}noue-{B}ombieri surfaces.
\newblock \emph{J. Funct. Anal.}, 2021.
\newblock URL \url{https://api.semanticscholar.org/CorpusID:235683210}.

\bibitem[Angella et~al.(2014)Angella, Dloussky, and Tomassini]{Angella2014OnBC}
D.~Angella, G.~Dloussky, and A.~Tomassini.
\newblock On {B}ott-{C}hern cohomology of compact complex surfaces.
\newblock \emph{Ann. Mat. Pura Appl. (1923 -)}, 195:\penalty0 199--217, 2014.
\newblock URL \url{https://api.semanticscholar.org/CorpusID:119171255}.

\bibitem[Apostolov and Dloussky(2007)]{apostolovanddloussky}
V.~Apostolov and G.~Dloussky.
\newblock Bihermitian metrics on hopf surfaces.
\newblock \emph{Mathematical Research Letters}, 15, 11 2007.
\newblock \doi{10.4310/MRL.2008.v15.n5.a1}.

\bibitem[Apostolov and Gualtieri(2007)]{ag:gkwithsplit}
V.~Apostolov and M.~Gualtieri.
\newblock Generalized {K}\"ahler manifolds, commuting complex structures, and split tangent bundles.
\newblock \emph{Comm. Math. Phys.}, 271:\penalty0 561 -- 575, 2007.

\bibitem[Apostolov and Streets(2017)]{Apostolov2017TheNG}
V.~Apostolov and J.~Streets.
\newblock The nondegenerate generalized {K}{\"a}hler {C}alabi–{Y}au problem.
\newblock \emph{J. Reine Angew. Math.}, 2021:\penalty0 1 -- 48, 2017.
\newblock URL \url{https://api.semanticscholar.org/CorpusID:119150656}.

\bibitem[Apostolov et~al.(2022)Apostolov, Fu, Streets, and Ustinovskiy]{apostolov2022generalized}
V.~Apostolov, X.~Fu, J.~Streets, and Y.~Ustinovskiy.
\newblock The generalized {K}\"ahler {C}alabi-{Y}au problem.
\newblock 2022.

\bibitem[Beauville(1998)]{Beauville}
A.~Beauville.
\newblock Complex manifolds with split tangent bundle.
\newblock \emph{arXiv: Algebraic Geometry}, 1998.
\newblock URL \url{https://api.semanticscholar.org/CorpusID:14042574}.

\bibitem[Bismut(1989)]{bismut:localindexthm}
J.-M. Bismut.
\newblock A local index theorem for non-{K}\"ahler manifolds.
\newblock \emph{Math. Ann.}, 284:\penalty0 681 -- 699, 1989.

\bibitem[Bismut(2013)]{bismut}
J.-M. Bismut.
\newblock \emph{Hypoelliptic {L}aplacian and {B}ott-{C}hern cohomology: A theorem of {R}iemann-{R}och-{G}rothendieck in complex geometry}.
\newblock Birkhauser/Springer, 2013.

\bibitem[Chu et~al.(2019)Chu, Tosatti, and Weinkove]{CTW}
J.~Chu, V.~Tosatti, and B.~Weinkove.
\newblock The {M}onge-{A}mp\`ere equation for non-integrable almost complex structures.
\newblock \emph{J. Eur. Math. Soc}, 21:\penalty0 1949--1984, 2019.
\newblock URL \url{https://doi.org/10.4171/JEMS/878}.

\bibitem[Collins(2016)]{collins:twistedtype}
T.~Collins.
\newblock {$C^{2,\alpha}$} estimates for nonlinear elliptic equations of twisted type.
\newblock \emph{Calc. Var. Partial Differential Equations}, 55\penalty0 (1), 2016.

\bibitem[Garcia-Fernandez and Streets(2021)]{gfs}
M.~Garcia-Fernandez and J.~Streets.
\newblock \emph{Generalized {R}icci flow}.
\newblock Am. Math. Soc., 2021.

\bibitem[Garcia-Fernandez et~al.(2023)Garcia-Fernandez, Jordan, and Streets]{gjs:nonkahler}
M.~Garcia-Fernandez, J.~Jordan, and J.~Streets.
\newblock Non-{K}{\"a}hler {C}alabi-{Y}au geometry and pluriclosed flow.
\newblock \emph{J. Math. Pures Appl.}, 177:\penalty0 329--367, 2023.
\newblock ISSN 0021-7824.
\newblock \doi{https://doi.org/10.1016/j.matpur.2023.07.002}.
\newblock URL \url{https://www.sciencedirect.com/science/article/pii/S0021782423000971}.

\bibitem[Gates et~al.(1978)Gates, Hull, and Ro\v{c}ek]{ghr:twistedmultiplet}
S.~Gates, C.~Hull, and M.~Ro\v{c}ek.
\newblock Twisted multiplets and new supersymmetric non-linear $\sigma$-models.
\newblock \emph{Comm. Pure Appl. Math.}, 31\penalty0 (3):\penalty0 339 -- 411, 1978.

\bibitem[Gauduchon(1977)]{gauduchon:nulle}
P.~Gauduchon.
\newblock Le th\`eor\`eme de l'exentricit\'e nulle.
\newblock \emph{Compt. Rend. Acad. Sci. Paris}, 285:\penalty0 387 -- 390, 1977.

\bibitem[Gauduchon and Ornea(1998)]{Gauduchon-Ornea}
P.~Gauduchon and L.~Ornea.
\newblock Locally conformally {K\"ahler} metrics on {Hopf} surfaces.
\newblock \emph{Ann. Inst. Fourier}, 48\penalty0 (4):\penalty0 1107--1127, 1998.
\newblock \doi{10.5802/aif.1651}.
\newblock URL \url{http://www.numdam.org/articles/10.5802/aif.1651/}.

\bibitem[Gilbarg and Trudinger(2001)]{gt}
D.~Gilbarg and N.~Trudinger.
\newblock \emph{Elliptic partial differential equations of second order}.
\newblock Springer, 2001.

\bibitem[Gualtieri(2011)]{g:gencompgeo}
M.~Gualtieri.
\newblock Generalized complex geometry.
\newblock \emph{Math. Ann.}, 174\penalty0 (1):\penalty0 75 -- 123, 2011.

\bibitem[Guan and Li(2010)]{guan2009complex}
B.~Guan and Q.~Li.
\newblock Complex {M}onge-{A}mp\`ere equations and totally real submanifolds.
\newblock \emph{Adv. Math.}, 225\penalty0 (3):\penalty0 1185--1223, 2010.
\newblock ISSN 0001-8708,1090-2082.
\newblock \doi{10.1016/j.aim.2010.03.019}.
\newblock URL \url{https://doi.org/10.1016/j.aim.2010.03.019}.

\bibitem[Hitchin(2003)]{hitchin:genCY}
N.~Hitchin.
\newblock Generalized {C}alabi-{Y}au manifolds.
\newblock \emph{Quart. J. Math.}, 54\penalty0 (3):\penalty0 281 -- 308, 2003.

\bibitem[Hull et~al.(2010)Hull, Lindstr{\"o}m, Ro\v{c}ek, von Unge, and Zabzine]{Hull2010GeneralizedCM}
C.~M. Hull, U.~Lindstr{\"o}m, M.~Ro\v{c}ek, R.~von Unge, and M.~Zabzine.
\newblock Generalized {C}alabi-{Y}au metric and generalized {M}onge-{A}mp{\`e}re equation.
\newblock \emph{J. High Energy Phys.}, 2010:\penalty0 1--23, 2010.
\newblock URL \url{https://api.semanticscholar.org/CorpusID:119636765}.

\bibitem[Ivanov and Papadopoulos(2001)]{IvanovPapa}
S.~Ivanov and G.~Papadopoulos.
\newblock Vanishing theorems and string backgrounds.
\newblock \emph{Classical Quantum Gravity}, 18\penalty0 (6):\penalty0 1089--1110, 2001.
\newblock ISSN 0264-9381.
\newblock URL \url{https://doi.org/10.1088/0264-9381/18/6/309}.

\bibitem[Jordan and Streets(2020)]{js:c3}
J.~Jordan and J.~Streets.
\newblock On a {C}alabi-type estimate for pluriclosed flow.
\newblock \emph{Adv. in Math.}, 366, 2020.

\bibitem[Kodaira(1975)]{Kodaira+1975+1471+1510}
K.~Kodaira.
\newblock \emph{On The Structure Of Compact Complex Analytic Surfaces, II}, pages 1471--1510.
\newblock Princeton University Press, Princeton, 1975.
\newblock ISBN 9781400869879.
\newblock \doi{doi:10.1515/9781400869879-013}.
\newblock URL \url{https://doi.org/10.1515/9781400869879-013}.

\bibitem[Mooney and Savin(2019)]{Mooney-Savin}
C.~Mooney and O.~Savin.
\newblock Regularity results for the equation $ u_{11}u_{22} = 1 $.
\newblock \emph{Discrete Contin. Dyn. Syst.}, 39\penalty0 (12):\penalty0 6865--6876, 2019.
\newblock ISSN 1078-0947.
\newblock \doi{10.3934/dcds.2019235}.
\newblock URL \url{https://www.aimsciences.org/article/id/30732454-76bd-4748-9b31-1c991c4671e1}.

\bibitem[Schweitzer(2007)]{Schweitzer2007AutourDL}
M.~Schweitzer.
\newblock Autour de la cohomologie de bott-chern.
\newblock \emph{arXiv: Algebraic Geometry}, 2007.
\newblock URL \url{https://api.semanticscholar.org/CorpusID:115165822}.

\bibitem[Shen and Smith(2022)]{Shen2022TheCE}
X.~S. Shen and K.~Smith.
\newblock The continuity equation on {H}opf and {I}noue surfaces.
\newblock \emph{Int. Math. Res. Not. IMRN}, 2022.
\newblock URL \url{https://api.semanticscholar.org/CorpusID:252199966}.

\bibitem[Streets(2014)]{Streets2014PluriclosedFO}
J.~Streets.
\newblock Pluriclosed flow on generalized {K}{\"a}hler manifolds with split tangent bundle.
\newblock \emph{J. Reine Angew. Math.}, 2014.
\newblock URL \url{https://api.semanticscholar.org/CorpusID:119612427}.

\bibitem[Streets(2016{\natexlab{a}})]{s:borninfeld}
J.~Streets.
\newblock Pluriclosed flow, {B}orn-{I}nfeld geometry, and rigidity results for generalized {K}{\"a}hler manifolds.
\newblock \emph{Comm. Partial Differential Equations}, 41\penalty0 (2):\penalty0 318--374, 2016{\natexlab{a}}.
\newblock \doi{10.1080/03605302.2015.1116560}.
\newblock URL \url{https://doi.org/10.1080/03605302.2015.1116560}.

\bibitem[Streets(2016{\natexlab{b}})]{s:globallygenerated}
J.~Streets.
\newblock Pluriclosed flow on manifolds with globally generated bundles.
\newblock \emph{Complex Manifolds}, 3\penalty0 (1), 2016{\natexlab{b}}.

\bibitem[Streets(2018)]{s:pcfongkwithsplit}
J.~Streets.
\newblock Pluriclosed flow on generalized {K}\"ahler manifolds with split tangent bundle.
\newblock \emph{J. Reine Angew. Math.}, 2018\penalty0 (739):\penalty0 241 -- 276, 2018.

\bibitem[Streets and Tian(2010)]{st:pluriclosed}
J.~Streets and G.~Tian.
\newblock A parabolic flow of pluriclosed metrics.
\newblock \emph{Int. Math. Res. Not.}, 2010\penalty0 (16):\penalty0 3101 -- 3133, 2010.

\bibitem[Streets and Tian(2012)]{STREETS2012366}
J.~Streets and G.~Tian.
\newblock Generalized {K}\"ahler geometry and the pluriclosed flow.
\newblock \emph{Nuclear Phys. B}, 858\penalty0 (2):\penalty0 366--376, 2012.
\newblock ISSN 0550-3213.
\newblock \doi{https://doi.org/10.1016/j.nuclphysb.2012.01.008}.
\newblock URL \url{https://www.sciencedirect.com/science/article/pii/S0550321312000211}.

\bibitem[Streets and Ustinovskiy(2021)]{su:gkrsolitons}
J.~Streets and Y.~Ustinovskiy.
\newblock Classification of generalized {K}{\"a}hler-{R}icci solitons on complex surfaces.
\newblock \emph{Comm. Pure Appl. Math.}, 74\penalty0 (9):\penalty0 1896--1914, 2021.
\newblock \doi{https://doi.org/10.1002/cpa.21947}.
\newblock URL \url{https://onlinelibrary.wiley.com/doi/abs/10.1002/cpa.21947}.

\bibitem[Streets and Warren(2016)]{sw:evanskrylov}
J.~Streets and M.~Warren.
\newblock Evans-{K}rylov estimates for a nonconvex {M}onge-{A}mp\`ere equation.
\newblock \emph{Math. Ann.}, 365:\penalty0 805 -- 834, 2016.

\bibitem[Tosatti and Weinkove(2010{\natexlab{a}})]{TW:MAonHerm}
V.~Tosatti and B.~Weinkove.
\newblock The complex {M}onge-{A}mp\`ere equation on compact {H}ermitian manifolds.
\newblock \emph{J. Amer. Math. Soc.}, 23:\penalty0 1187--1195, 2010{\natexlab{a}}.
\newblock URL \url{https://doi.org/10.1090/S0894-0347-2010-00673-X}.

\bibitem[Tosatti and Weinkove(2010{\natexlab{b}})]{TWBalancedandHerm}
V.~Tosatti and B.~Weinkove.
\newblock {Estimates for the Complex Monge-Ampère Equation on Hermitian and Balanced Manifolds}.
\newblock \emph{Asian J. Math.}, 14\penalty0 (1):\penalty0 19 -- 40, 2010{\natexlab{b}}.

\bibitem[Tosatti and Weinkove(2017)]{TW:MAforpluri}
V.~Tosatti and B.~Weinkove.
\newblock The {M}onge-{A}mp\`ere equation for $(n-1)$-plurisubharmonic functions on a compact {K}\"ahler manifold.
\newblock \emph{J. Amer. Math. Soc.}, 30:\penalty0 311--346, 2017.
\newblock URL \url{https://doi.org/10.1090/jams/875}.

\bibitem[Tosatti and Weinkove(2019)]{TW:HermformsMA}
V.~Tosatti and B.~Weinkove.
\newblock {H}ermitian metrics, $(n-1,n-1)$-forms, and {M}onge-{A}mp\`ere equations.
\newblock \emph{J. Reine Angew. Math.}, 755:\penalty0 67--101, 2019.
\newblock URL \url{https://doi.org/10.1515/crelle-2017-0017}.

\bibitem[Tosatti and Weinkove(2022)]{TosattiWeinkoveCRF}
V.~Tosatti and B.~Weinkove.
\newblock The {C}hern-{R}icci flow.
\newblock \emph{Atti Accad. Naz. Lincei Cl. Sci. Fis. Mat. Natur.}, 33\penalty0 (1):\penalty0 73--107, 2022.

\bibitem[Tricerri(1982)]{MR0706055}
F.~Tricerri.
\newblock Some examples of locally conformal {K}\"{a}hler manifolds.
\newblock \emph{Rend. Sem. Mat. Univ. Politec. Torino}, 40\penalty0 (1):\penalty0 81--92, 1982.
\newblock ISSN 0373-1243.

\bibitem[Wang et~al.(2016)Wang, Yang, and Zheng]{Wang2016OnBF}
Q.~Wang, B.~Yang, and F.~Zheng.
\newblock On {B}ismut flat manifolds.
\newblock \emph{Trans. Amer. Math. Soc.}, 2016.
\newblock URL \url{https://api.semanticscholar.org/CorpusID:55359054}.

\bibitem[Yau(1978)]{yau:cyproof}
S.-T. Yau.
\newblock On the {R}icci curvature of a compact {K}\"ahler manifold and the complex {M}onge-{A}mp\`ere equation, {I}.
\newblock \emph{Comm. Pure Appl. Math.}, 31\penalty0 (3):\penalty0 339 -- 411, 1978.

\bibitem[Ye(2023)]{Ye2022BismutEM}
Y.~Ye.
\newblock Bismut {E}instein metrics on compact complex manifolds, 2023.

\bibitem[Zheng(2019)]{Zhengsurvey}
F.~Zheng.
\newblock Some recent progress in non-{K}\"ahler geometry.
\newblock \emph{Sci. China Math.}, 62, 2019.
\newblock URL \url{https://doi.org/10.1007/s11425-019-9528-1}.

\end{thebibliography}

\end{document}